\documentclass[11pt, a4paper, oneside]{amsart}
\usepackage{amsrefs,mathrsfs,amssymb}
\usepackage[T1]{fontenc}
\usepackage[utf8]{inputenc}
\usepackage[body={16cm, 24cm}]{geometry}
\usepackage{graphicx,mathrsfs,amssymb,color,comment}
\newtheorem{theorem}{\textbf{Theorem}}[section]
\newtheorem{proposition}[theorem]{\textbf{Proposition}}
\newtheorem{lemma}[theorem]{\textbf{Lemma}}
\newtheorem{corollary}[theorem]{\textbf{Corollary}}

\theoremstyle{definition}
\newtheorem{definition}[theorem]{\textbf{Definition}}

\theoremstyle{remark}

\newtheorem{example}[theorem]{Example}
\newtheorem{remark}[theorem]{Remark}

\allowdisplaybreaks
\numberwithin{equation}{section}
\let\inf\undefined\DeclareMathOperator*{\inf}{inf\vphantom{p}}
\let\min\undefined\DeclareMathOperator*{\min}{min\vphantom{p}}
\let\max\undefined\DeclareMathOperator*{\max}{max\vphantom{p}}

\title[Generalized differentiation in Wasserstein]{Generalized differentiation in Wasserstein space and application to multiagent control problem}

\author[R. Capuani]{Rossana Capuani}
\address{\hspace{-0.5em}\begin{tabular}{ll}Rossana Capuani:&Department of Mathematics,\\& The University of Arizona\\
& 617 N Santa Rita Ave, Tucson, AZ 85721, United States.\end{tabular}}
\email{rossanacapuani@arizona.edu}

\author[A. Marigonda]{Antonio Marigonda}
\address{\hspace{-0.5em}\begin{tabular}{ll}Antonio Marigonda:&Department of Computer Science,\\& University of Verona\\
&Strada Le Grazie 15, I-37134 Verona, Italy.\end{tabular}}
\email{antonio.marigonda@univr.it}

\author[M. Quincampoix]{Marc Quincampoix}
\address{\hspace{-0.5em}\begin{tabular}{ll}Marc Quincampoix:&Laboratoire de Math\'ematiques de Bretagne Atlantique,\\ &Univ Brest, CNRS UMR 6205, \\
&6, avenue Victor Le Gorgeu,\\ &29200 Brest, France.\end{tabular}}
\email{marc.quincampoix@univ-brest.fr}

\date{\today}

\begin{document}

\begin{abstract}
Several concepts of generalized differentiation in Wasserstein space have been proposed in order to deal with the intrinsic nonsmoothness arising in the context of optimization problems in Wasserstein spaces. 
In this paper we introduce a concept of \emph{admissible variation} encompassing some of the most popular definitions as special cases, and using it to derive a comparison principle for viscosity solutions of 
an Hamilton-Jacobi-Bellman equation following from an optimal control of a multiagent systems.
\end{abstract}

\maketitle

\section{Introduction} 
This paper concerns the differentiation of real-valued functions defined on the Wasserstein space $\mathscr P(\mathbb T^d)$, where $\mathbb T ^d$ denotes the $d$-dimensional torus and $\mathscr P(\mathbb T^d)$ the set of probability measures on  $\mathbb T ^d$.
\par\medskip\par
This topic has been extensively investigated in recent years. Major developments have occurred within the framework of mean field games (see, e.g., \cites{AMQ, BonF, CDLL, LL}), in multi-agent control problems (for recent contributions, see \cites{CQ, CMQ, CLOS, JMQ, JMQ21, MQ}), and in the theory of gradient flows (cf. \cite{AGS}, \cite{San}). Moreover, several approaches have been proposed in \cites{AGS, AJZ, BadF, CDLL, Jimenez}, particularly in connection with the formulation of appropriate notions of viscosity solutions for Hamilton–Jacobi equations in the Wasserstein space (see \cites{G1, G2}) and the study of their well-posedness.
\par\medskip\par
In order to address a wide variety of problems, multiple non-equivalent notions of derivative have been introduced. As a result, several competing definitions now coexist. Although each appears well-suited within the framework that originated it, a direct comparison between these notions remains challenging--even under suitable assumptions of smoothness (to be made precise later).
The main difficulty in consistently defining a notion of derivative in the Wasserstein space stems from the lack of linear structure of $\mathscr P(\mathbb T^d)$ endowed with the Wasserstein distance. Since this space is not a vector space, the classical Fréchet or Gâteaux derivatives--whose definitions rely on vector addition and scalar multiplication--cannot be directly applied.
\par\medskip\par
Recall that for a smooth function $\phi:\mathbb R^d \to \mathbb R$, the classical directional derivative of $\phi$ at $x \in \mathbb R^d$ in the direction $v \in \mathbb R^d \setminus \{0\}$ is obtained by considering the variations
\[x_{v,t} := x + t v,\]
and passing to the limit
\[\lim_{t \to 0^+} \frac{\phi(x + t v) - \phi(x)}{t},\]
which amounts to examining increasingly small perturbations of $x$ along the direction $v$. 
In contrast, since $\mathscr P(\mathbb T^d)$ is not a normed vector space, the variation $t \mapsto x + t v$ has no immediate analogue. 
Consequently, an alternative notions of ``variation'' must be devised to define differentiability meaningfully in the space of probability measures.
\par\medskip\par
In order to recover a linear structure, a popular approach introduced by Lions consists in \emph{lifting} functions of measures defined on $\mathscr P_2(\mathbb T^d)$ to functions on a Hilbert space of random variables. 
More precisely, a probability measure $m \in \mathscr P_2(\mathbb T^d)$ is identified with a random variable $X \in L^2(\Omega; \mathbb R^d)$ such that $\mathrm{Law}(X) = m$, and so the lifted functional $\tilde \Phi : L^2(\Omega; \mathbb R^d) \to \mathbb R$ can be defined by
\[\tilde \Phi(X) := \Phi(\mathrm{Law}(X)).\]
This allows to exploit the linear structure of $L^2(\Omega; \mathbb R^d)$ in order to define a differential structure. 
\par\medskip\par
However, the derivative obtained in this way generally depends on the choice of the probability space and on the representative $X$, unless additional structural conditions are imposed. 
Proving that the derivative is well defined and independent of these choices is a delicate matter. 
Moreover, even when a derivative exists in the lifted space, it may not admit a simple or intrinsic representation at the level of probability measures.
\par\medskip\par
On the other hand, relying solely on the metric structure---namely, on the notions of metric slopes and metric derivatives---is largely insufficient in many concrete situations. 
Informally, these metric notions provide directions of steepest descent and can, in a generalized sense, be used to define metric subdifferentials; however, they do not yield a structure comparable to that of linear maps acting on tangent vectors. 
\par\medskip\par
In addition to these difficulties, it must be noticed also that functionals on measures are typically \emph{nonlocal}, in the sense that their values depend on the entire measure rather than on its behavior at a single point. 
As a consequence, any pointwise interpretation of derivatives on the underlying space becomes problematic, since the derivative itself may inherit the same nonlocal character.
\par\medskip\par
In this paper, we introduce and analyze a general concept of \emph{variation} associated with the problem of differentiating a functional $U : \mathscr P(\mathbb T^d) \to \mathbb R$. 
We show that this notion is sufficiently broad to encompass several of the most widely used definitions of derivatives in the Wasserstein space. 
Within this unified framework, we establish that---under natural smoothness assumptions---various existing notions of differentiability in the literature actually coincide. 
\par\medskip\par
The concept of variation introduced in the first part of the paper is further developed in the second part, where it serves as the foundation for defining appropriate notions of sub- and superdifferentials. 
These, in turn, allow us to establish a comparison principle for viscosity solutions of first-order Hamilton--Jacobi equations in the Wasserstein space.
\par\medskip\par
\par\medskip\par
We apply this latter result to characterize the value function associated with a suitable mean field control problem of Bolza type. 
In this framework, a central planner seeks to minimize a Bolza-type cost functional depending on the configurations of a continuum of agents (the \emph{followers}) and a distinguished agent (the \emph{leader}), whose dynamics are coupled through an asymmetric interaction.
\par\medskip\par
The paper is organized as follows. 
In Section~\ref{sec:admvar}, we recall some preliminary definitions and results concerning the Wasserstein space and introduce a suitable notion of admissible variations, together with a comparison to related concepts in the literature. 
Section~\ref{sec:diff} is devoted to the definition of differentiation along variations and to the proof of an equivalence result under appropriate smoothness assumptions. 
In Section~\ref{sec:model}, we present the multi-agent model and formulate the mean field control problem under consideration. 
Finally, in Section~\ref{sec:HJB}, we establish a comparison principle for the Hamilton--Jacobi--Bellman equation associated with the mean field control problem, thereby characterizing the corresponding value function.

\section{Admissible variations on the Wasserstein space.}\label{sec:admvar}

\subsection{Preliminaries} 
Here we recall some basic facts on Wasserstein spaces, referring to \cites{AGS, Villani, Santambrogio} for a more detailed presentation.
\par\medskip\par
Given a complete separable metric space $(X,d)$, denote by $\mathscr P(X)$ the space of Borel probability measures on $X$, which can be identified with a convex subset of $(C^0_b(X))'$, the topological dual of 
real-valued continuous bounded functions on $X$. Given two sets $X_1,X_2$, we denote by $\mathrm{pr}_i:X_1\times X_2\to X_i$, $i=1,2$, the map defined by setting $\mathrm{pr}_i(x_1,x_2)=x_i$, $i=1,2$.
Given another metric space $Y$, a Borel map $f:X\to Y$, and $\mu\in\mathscr P(X)$, the pushforward measure $f\sharp\mu\in\mathscr P(Y)$ is defined by $f\sharp\mu(A)=\mu(f^{-1}(A))$ for every Borel subset of $Y$.
Given $\mu\in\mathscr P(X)$, $\mu'\in\mathscr P(Y)$, a transport plan $\pi$ between $\mu$ and $\mu'$ is an element of $\mathscr P(X\times Y)$
satisfying $\mathrm{pr}_1\sharp\pi=\mu$, $\mathrm{pr}_2\sharp\pi=\mu'$. The set of transport plans between $\mu$ and $\mu'$ is denoted by $\Pi(\mu,\mu')$.
Given $\pi\in\Pi(\mu_1,\mu_2)\subseteq\mathscr P(X\times X)$, $q\ge 1$ we set
\[W^{\pi}_q(\mu_1,\mu_2):=\left[\iint_{X\times X}d^q(x,y)\,d\pi(x,y)\right]^{1/q}.\]
The $q$-Wasserstein distance is defined by \[W_q(\mu_1,\mu_2):=\inf_{\pi\in\Pi(\mu_1,\mu_2)}W^{\pi}_q(\mu_1,\mu_2),\] we know \cite{Villani} that the above minimum is achieved and  the plans $\pi$ realizing the infimum are called \emph{optimal transport plans} between $\mu_1$ and $\mu_2$.
\par
We recall the following version of the desintegration theorem (Theorem 5.3.1 in \cite{AGS}).

\begin{proposition}[Disintegration]\label{thm:disintegration}
Given a measure $\mu\in\mathscr P(X \times Y)$, there exists a Borel family of probability measures $\{\mu_x\}_{x\in X}\subseteq \mathscr P(Y)$,
uniquely defined for $\mathrm{pr}_1\sharp \mu$-a.e. $x\in X$, such that for any Borel map $\varphi:X\times Y\to [0,+\infty]$ we have
\[\int_{X\times Y}\varphi(x,y)\,d\mu(x,y)=\int_X \left[\int_{Y}\varphi(x,y)\,d\mu_x(y)\right]d(\mathrm{pr}_1\sharp \mu)(x).\]
We will write shortly $\mu=(\mathrm{pr}_1\sharp \mu)\otimes \mu_x$.
\end{proposition}

In this paper we focus on the Wasserstein space on the torus $\mathscr P(\mathbb T^d)$ equipped with the topology of weak (narrow) convergence of measures 
(which is equivalent to the topology generated by $W_q$ since $\mathbb T^d$ is compact, see \cite{Santambrogio}).

Given a compact interval $I$ of $\mathbb R$, we denote by $\Gamma_I=C^0(I;\mathbb T^d)$ the space of continuous curves defined on $I$ with values in $\mathbb T^d$,
endowed with the usual sup-norm, and define for every $t\in I$ the linear and continuous map $\hat e_t:\Gamma_I\to\mathbb R^d$ by setting $\hat e_t(\gamma)=\gamma(t)$.

Given a family of probability measures $\boldsymbol\mu=\{\mu_t\}_{t\in I}$, and a Borel vector field $v=v_t(x)$ such that 
\begin{equation} \label{eq:intcond}\int_I\int_{\mathbb T^d}|v_t(x)|^q\,d\mu_t(x)\,dt<+\infty, \end{equation}
we say that $\boldsymbol\mu$ is a solution the \emph{continuity equation} driven by $v$ if
\begin{equation} \label{eq:conteq}\partial_t\mu_t+\mathrm{div}(v_t\mu_t)=0\end{equation}
in the sense of distributions. In particular, the map $t\mapsto \mu_t$ is continuous.
\par\medskip\par
According to the \emph{superposition principle} (Theorem 8.2.1 in \cite{AGS}), for every solution of \eqref{eq:conteq} there exists $\boldsymbol\eta\in\mathscr P(\Gamma_{I})$ 
such that $\mu_t=\hat e_t\sharp\boldsymbol\eta$ for all $t\in I$, and $\boldsymbol\eta$ is supported on the integral solutions of $\dot\gamma(t)=v_t\circ\gamma(t)$.
Conversely, given any Borel vector field $v=v_t(\cdot)$, $\boldsymbol\eta\in\mathscr P(\Gamma_{I})$, supported on the integral solutions of $\dot\gamma(t)=v_t\circ\gamma(t)$, and defined $\mu_t=\hat e_t\sharp\boldsymbol\eta$ for all $t\in I$,
we have that if \[\int_I \int_{\mathbb R^d}|v_t(x)|^q\,d\mu_t(x)\,dt<+\infty,\] 
then $\boldsymbol\mu:=\{\mu_t\}$ solves \eqref{eq:conteq}. In this case we will say that $\boldsymbol\eta$ is a representative of $\boldsymbol\mu$ according to the superposition principle.
\par\medskip\par
For any $\mu\in\mathscr P_2(X)$ we define $\displaystyle\mathrm{m}_2(\mu)=\int_{X}|x|^2\,d\mu(x)$.
Let $\mathbb X$ be a metric space, $t_0<t_1<t_2$, $\gamma_1\in C^0([t_0,t_1],\mathbb X)$, $\gamma_2\in C^0([t_1,t_2],\mathbb X)$. 
We define $\gamma_1\oplus\gamma_2:[t_0,t_2]\to \mathbb X$
by setting $$\gamma_1\oplus\gamma_2(t)=\begin{cases}\gamma_1(t),&t\in [t_0,t_1],\\ \gamma_2(t),&t\in ]t_1,t_2].\end{cases}$$
Notice that $\gamma_1\oplus\gamma_2\in C^0([t_0,t_2],\mathbb X)$ if and only if $\gamma_1(t_1)=\gamma_2(t_1)$.

\subsection{Admissible $q$-variations}
Let $q \geq 1$ be fixed. Throughout  the paper we consider the conjugate exponent $p \geq 1$ of $q$, namely the unique $p$ satisfying $\frac1p + \frac1q =1$. 

We introduce now a new concept of variations

\begin{definition}[Admissible variations ]\label{def:adm-var}
Let $\mu\in\mathscr P(\mathbb T^d)$, $T >0$. A Borel family of plans $\boldsymbol\pi:=\{\pi_t\}_{t\in [0,T] }\subseteq \mathscr P(\mathbb T^d\times\mathbb T^d)$ is a \emph{family of admissible $q$-variations} from $\mu$ if 
\begin{itemize}
\item[]
\item[a.] $\mathrm{pr}_1\sharp\pi_t=\mu$ for all $t\in [0,T]$,
\item[]
\item[b.] $\pi_t\rightharpoonup(\mathrm{Id},\mathrm{Id})\sharp\mu$ as $t\to 0^+$,
\item[]
\item[c.] there a constant $C>0$ such that $\displaystyle\limsup_{t\to 0^+}\dfrac{W^{\pi_t}_q(\mu,\mathrm{pr}_2\sharp\pi_t)}{t}\le C$.
\end{itemize}
\end{definition}

\begin{remark}
Condition a. and b. of Definition \ref{def:adm-var} expresses the natural fact that the base point $\mu$ is fixed in the limit operation, and that the amplitude of the variation 
tends to vanish in a suitable weak sense as $t\to 0^+$. Condition c. states that the speed of this latter process is requested to occurr at a linear rate, 
as in the classical directional derivative.
\end{remark}

The following stability properties of variations are a straightforward consequence of Definition \ref{def:adm-var}.
\begin{proposition}  
Let $\{\pi_t\}_{t\in [0,T] }$   and $\{\pi'_t\}_{t\in [0,T] }$ be two  families of admissible $q$-variations from $\mu$. Then
\begin{itemize}
\item For any $ s \in [0,1] $,  $\{s \pi _t + (1-s) \pi'_t\}_{t\in [0,T] }$  is a family of admissible $q$-variations from $\mu$.
\item Let $I \cup J =[0,T] $ be a partition of $[0,T]$. The family  $\{1_I(t) \pi _t + 1_J(t) \pi'_t\}_{t\in [0,T] }$ is a family of admissible $q$-variations from $\mu$.
\end{itemize}
\end{proposition}
Now we will show that this notion of variations is stable by "perturbations by small plans" (see proposition \ref{prop:perturb} below).  Let us first recall the following composition of plans with a common margin.
\begin{definition}[Composition of plans]
Let $\pi_{12}\in\mathscr P(X_1\times X_2)$, $\pi_{23}\in\mathscr P(X_2\times X_3)$ be satisfying $\mathrm{pr}_2\sharp\pi_{12}=\mathrm{pr}_1\sharp\pi_{23}$.
Set $\mu_i\in \mathscr P(X_i)$, $i=1,2,3$, by $\mu_1=\mathrm{pr}_1\sharp\pi_{12}$, $\mu_2=\mathrm{pr}_2\sharp\pi_{12}$, $\mu_3=\mathrm{pr}_2\sharp\pi_{23}$.
Then we define $\pi_{12}\circ \pi_{23}\in\mathscr P(X_1\times X_3)$ by letting for all $\varphi\in C^0_b(X_1\times X_3)$
\[\iint_{X_1\times X_3}\varphi(x_1,x_3)\,d(\pi_{12}\circ \pi_{23})(x_1,x_3)=\int_{X_2}\left[\iint_{X_1\times X_3}\varphi(x_1,x_3)\,d\pi^{x_2}_{12}(x_1)d\pi^{x_2}_{23}(x_3)\right]\,d\mu_2(x_2),\]
where we disintegrate $\pi_{12}=\mu_2\otimes\pi^{x_2}_{12}$ and $\pi_{23}=\mu_2\otimes \pi^{x_2}_{23}$.
\end{definition}

\begin{remark} \label{rem:comp}
Notice that we have $\mathrm{pr}_i\sharp(\pi_{12}\circ \pi_{23})=\mu_i$, $i=1,3$, and
\[W_p^{\pi_{12}\circ\pi_{23}}(\mu_1,\mu_3)\le W_p^{\pi_{12}}(\mu_1,\mu_2)+W_p^{\pi_{23}}(\mu_2,\mu_3).\]
\end{remark}

As an immediate consequence of Remark \ref{rem:comp}, we deduce the following stability result of variations
\begin{proposition}  \label{prop:perturb}
Let $\{\pi_t\}_{t\in [0,T]}$ be a family of admissible
$q$-variations from $\mu\in\mathscr P(\mathbb T^d)$.
Suppose that $\{\hat \pi_t\}_{t\in [0,T]}\subseteq \mathscr P(\mathbb
T^d\times\mathbb T^d)$ satisfies
\begin{align*}
\mathrm{pr}_2\sharp\pi_t=\mathrm{pr}_1\sharp\hat \pi_t\textrm{ for all
}t\in [0,T],&&\lim_{t\to
0^+}\dfrac{W^{\hat\pi_t}(\mathrm{pr}_1\sharp\hat
\pi_t,\mathrm{pr}_2\sharp\hat \pi_t)}{t}=0.
\end{align*}
Then $\{\pi_t\circ\hat\pi_t\}_{t\in [0,T]}$ is a family of admissible
$q$-variations from $\mu\in\mathscr P(\mathbb T^d)$.
\end{proposition}

\subsection{Comparison with other type of variations}

We first recall other concept of variations, some of which appeared already in the literature
\begin{definition}\label{def:adm-var-ex}
We will consider the following class of variations.
\begin{enumerate}
\item \emph{Transport map variations (cf e.g. \cite{CQ})}: given a Borel vector field $\phi(\cdot)$ and $\mu\in\mathscr P(\mathbb T^d)$, we define $\pi_t=(\mathrm{Id},(\mathrm{Id}+t\phi))\sharp\mu$.
\item[]
\item \emph{Lagrangian variations}: given  $\mu\in\mathscr P(\mathbb T^d)$ and a measure $\boldsymbol\eta\in\mathscr P(\Gamma_I)$ with $ \hat e_0 \sharp \eta = \mu $, we define $\pi_t=(\hat e_0,\hat e_t)\sharp\boldsymbol\eta$.
\item[]
\item \emph{Flat variations (cf e.g. \cite{CDLL})}: given $\pi\in\mathscr P(\mathbb T^d\times\mathbb T^d)$, set $\mu:=\mathrm{pr}_1\sharp\pi$, we define  
$\pi_t=(\mathrm{Id},\mathrm{Id})\sharp\mu+t[\pi-(\mathrm{Id},\mathrm{Id})\sharp\mu]$. 
\item[]

\item \emph{Eulerian variations}: let $\boldsymbol\mu=\{\mu_t\}_{t\in I}$ be a solution of the continuity equation \eqref{eq:conteq} related to a Borel vector field satisfying the integrability condition \eqref{eq:intcond},
and let $\boldsymbol\eta\in\mathscr P(\Gamma_I)$ be its representative according to the superposition principle. Then we define $\pi_t=(\hat e_0,\hat e_t)\sharp\boldsymbol\eta$. 
\end{enumerate}
\end{definition}

Concerning the less immediate relationships between these family of variations, the main ingredient is Theorem 8.3.1 in \cite{AGS}, linking absolutely continuous curves in the space of probability measures and
solutions of the continuity equation.

\begin{lemma}\label{lemma:intass}\leavevmode\par
\begin{enumerate}
\item Eulerian variations agree with the Lagrangian variations constructed on the representative of solutions of the continuity equations.
\item Transport map variations are a special case of Lagrangian variations, and if $\phi\in L^{p}_{\mu}(\mathbb T^d)$, then the related transport map variation is a special case of Eulerian variations.
\item Flat variations are a special case of Eulerian variations.
\end{enumerate}
\end{lemma}
\begin{proof}\leavevmode\par
1. Trivial from the definition.
\par\medskip\par
2. Given a Borel map $\phi:\mathbb T^d\to\mathbb T^d$ and $\mu\in\mathscr P(\mathbb T^d)$, we first notice that the map $(x,v)\mapsto \gamma_{x,v}$, 
defined by $\gamma_{x,v}(t)=x+tv$ for all $t\in I$, is continuous from $\mathbb T^d\times\mathbb T^d$ to $\Gamma_I$,
and that the map $x\mapsto (x,\phi(x))$ is Borel. Thus we have that $x\mapsto \gamma_{x,\phi(x)}$ is Borel, 
and we can define $\boldsymbol\eta=\mathrm{pr}_2\sharp(\mu\otimes\delta_{\gamma_{x,\phi(x)}})$,  i.e., for all $\varphi\in C^0_b(\Gamma_I)$
\[\int_{\Gamma_I}\varphi(\gamma)\,d\boldsymbol\eta(\gamma)=\int_{\mathbb T^d}\varphi(\gamma_{x,\phi(x)})\,d\mu(x).\]
With this definition, we have $(\hat e_0,\hat e_t)\sharp\boldsymbol\eta=(\mathrm{Id},\mathrm{Id}+t\phi)\sharp\mu$ for all $t\in I$.
Furthermore, set $\mu_t=\hat e_t\sharp\boldsymbol\eta$, we have 
\[W_p(\mu_t,\mu_s)\le\left(\int_{\mathbb T^d\times\mathbb R^d}|x-y|^p\,d[(\hat e_t,\hat e_s)\sharp\boldsymbol\eta](x,y)\right)^{1/p}=|t-s|\left(\int_{\mathbb T^d}|\phi(x)|\,d\mu(x)\right)^{1/p}.\]
In particular, if $\phi\in L^{p}_{\mu}(\mathbb T^d)$ we have that $t\mapsto\mu_t$ is $W_p$-Lipschitz continuous with constant less than $\|\phi\|_{L^p_{\mu}}$.
According to Theorem 8.3.1 in \cite{AGS}, we have that there exists a Borel vector field $v=v_t(\cdot)$ such that $\boldsymbol\mu=\{\mu_t\}_{t\in I}$ solves the continuity 
equation \eqref{eq:conteq} driven by $v$. Furthermore, $\|v_t\|_{L^p_{\mu_t}}\le \mathrm{Lip}(\boldsymbol\mu)\le \|\phi\|_{L^p_{\mu}}$ for a.e. $t\in I$, so the integrability condition \eqref{eq:intcond} is satisfied.
\par\medskip\par
3. Let $\pi\in\Pi(\mu,\nu)$ and define $\pi_t=(\mathrm{Id},\mathrm{Id})\sharp\mu+t[\pi-(\mathrm{Id},\mathrm{Id})\sharp\mu]$, $\mu_t=\mathrm{pr}_2\sharp\pi_t$, $t\in I$, $\boldsymbol\mu=\{\mu_t\}_{t\in I}$.
Then for every $\varphi\in\mathrm{Lip}(\mathbb T^d)$ it holds
\begin{align*}
\int_{\mathbb T^d} \varphi(y)\,d(\mathrm{pr}_2\sharp \pi_t)(y)&-\int_{\mathbb T^d} \varphi(y)\,d(\mathrm{pr}_2\sharp \pi_s)(y)=\\
=&\iint_{\mathbb T^d\times\mathbb T^d} \varphi(y)\,d\pi_t(x,y)-\iint_{\mathbb T^d\times\mathbb T^d} \varphi(y)\,d\pi_s(x,y)\\
=&(t-s)\left[\iint_{\mathbb T^d\times\mathbb T^d}\varphi(y)\,d\pi(x,y)-\iint_{\mathbb T^d\times\mathbb T^d} \varphi(y)\,d[(\mathrm{Id},\mathrm{Id})\sharp\mu](x,y)\right]\\
=&(t-s)\left[\int_{\mathbb T^d}\varphi(y)\,d\nu(y)-\int_{\mathbb T^d\times\mathbb T^d} \varphi(y)\,d\mu(y)\right].
\end{align*}
By passing to the supremum on $\varphi$ satisfying $\mathrm{Lip}\,\varphi\le 1$, and recalling Kantorovich duality (see e.g. Theorem 5.1 in \cite{Santambrogio}), we have
\[W_1(\mu_t,\mu_s)\le |t-s| W_1(\mu,\nu).\]
In particular, the curve $t\mapsto \mu_t$ is Lipschitz continuous with respect to $W_1$-metric, with $\mathrm{Lip}(\boldsymbol\mu)\le W_1(\mu,\nu)$. 
By Theorem 8.3.1 in \cite{AGS}, we have that there exists a Borel vector field $v=v_t(\cdot)$
such that $\boldsymbol\mu=\{\mu_t\}_{t\in I}$ solves the continuity equation \eqref{eq:conteq} driven by $v$. Furthermore, $\|v_t\|_{L^1_{\mu_t}}\le \mathrm{Lip}(\boldsymbol\mu)\le W_1(\mu,\nu)$ for a.e. $t\in I$,
so the integrability condition \eqref{eq:intcond} is satisfied.
\end{proof}

The following proposition shows that all the variations considered in Definition \ref{def:adm-var-ex} under suitable assumptions 
share the essential properties requested in the Definition \ref{def:adm-var} of admissible variations.

\begin{proposition}
The following variations are admissible variations according to Definition \ref{def:adm-var}
\begin{enumerate}
\item Transport map variations generated by $\phi\in L^q_{\mu}(\mathbb T^d)$;
\item Flat variations generated by $\pi\in \mathscr P(\mathbb T^d\times\mathbb T^d)$; 
\item Eulerian variations.
\end{enumerate}
\end{proposition}
\begin{proof}
According to Lemma \ref{lemma:intass}, it would be enough to prove $(3)$, however we provide different (and more direct) proof also for $(1)$ and $(2)$.
Without loss of generality, assume that $I=[0,1]$.
In all the three considered cases, properties a. and b. of Definition \ref{def:adm-var} are immediate,
in particular for $t\in ]0,1]$ sufficiently small it holds $W^{\pi_t}_q(\mu,\mu_t)\le 1$.
Concerning property c., set $\mu_t=\mathrm{pr}_2\sharp\mu$, and $1\le p\le \infty$ such that $1/p+1/q=1$, we have
\begin{enumerate}
\item in the case of transport map variations, for any $t\in ]0,1]$ it holds
\[\dfrac{1}{t}W^{\pi_t}_q(\mu,\mu_t)=\dfrac{1}{t}\left[\iint_{\mathbb T^d\times\mathbb T^d}|t\phi(x)|^q\,d\mu(x)\right]^{1/q}=\|\phi\|_{L^q_\mu}.\]
\item in the case of flat variations, we have
\begin{align*}
\dfrac{1}{t}W^{\pi_t}_q(\mu_1,\mu_2)=&\dfrac{1}{t}\left[\iint_{\mathbb T^d\times\Gamma_{[0,1]}}\left|y-x\right|^q\,d\pi_t(x,y)\right]^{1/q}
=\dfrac{t^{1/q}}{t}\left[\iint_{\mathbb T^d\times\Gamma_{[0,1]}}\left|y-x\right|^q\,d\pi(x,y)\right]^{1/q}\\ =&t^{1/p} W^\pi_q(\mu,\mu').
\end{align*}
\item in the case of Eulerian variation generated by a solution $\boldsymbol\mu=\{\mu_t\}_{t\in I}$ of the continuity equation \eqref{eq:conteq} driven by a Borel vector field $v$
satisfying \eqref{eq:intcond}, we have for all $t\in ]0,1]$ 
\begin{align*}
\dfrac{W^{\pi_t}_q(\mu,\mu_t)}{t}=&\dfrac{1}{t}\left[\iint_{\mathbb T^d\times\Gamma_{[0,1]}}\left|\gamma(t)-\gamma(0)\right|^q\,d\boldsymbol\eta(\gamma)\right]^{\frac{1}{q}}
=\dfrac{1}{t}\left[\iint_{\mathbb T^d\times\Gamma_{[0,1]}}\left|\int_0^t |v_s(\gamma(s))|\,ds\right|^q\,d\boldsymbol\eta(\gamma)\right]^{\frac{1}{q}}\\
\le&\dfrac{1}{t}\left[\iint_{\mathbb T^d\times\Gamma_{[0,1]}}\left|\left(\int_0^1 |v_s(\gamma(s))|^q\,ds\right)^{\frac{1}{q}} t^{\frac{1}{p}}\right|^q\,d\boldsymbol\eta(\gamma)\right]^{\frac{1}{q}}\\
=&\dfrac{t^{\frac{1}{p}}}{t}\left[\int_0^1 \iint_{\mathbb T^d\times\Gamma_{[0,1]}}|v_s(\gamma(s))|^q\,d\boldsymbol\eta(\gamma)\,ds\right]^{\frac{1}{q}}
=\dfrac{t^{\frac{1}{p}}}{t}\left[\int_0^1 \int_{\mathbb T^d}|v_s(x)|^q\,d\mu_t(x)\,ds\right]^{\frac{1}{q}}.
\end{align*}
\end{enumerate}
Thus in all the considered cases, property c. is satisfied. The proof is complete.
\end{proof}

We show now that Lagrangian variations are admissible according to Definition \ref{def:adm-var} under some conditions.

\begin{lemma}
Let $\boldsymbol\xi\in\mathscr P(\Gamma_I)$ be supported on $AC$ curves, and set $\mu_t=\hat e_t\sharp\boldsymbol\xi$.
Suppose that $\tau\mapsto \int_{\Gamma_I}|\dot\gamma|(\tau)\,d\boldsymbol\xi(\gamma)$ belongs to $L^1(I)$. Then, defined $\mu_t=\hat e_t\sharp\boldsymbol\xi$ for all $t\in I$,
there exists a Borel vector field $v=v_t(x)$ satisfying $\eqref{eq:intcond}$ such that $\boldsymbol\mu=\{\mu_t\}_{t\in I}$ solves \eqref{eq:conteq}.
In particular, the Lagrangian variation generated by $\boldsymbol\xi$ is actually an Eulerian variation.
\end{lemma}
\begin{proof}
For all $\varphi\in C^1_c(\mathbb T^d)$ it holds
\begin{align*}
\left|\int_{\Gamma_I}\varphi(\gamma(t))\,d\boldsymbol\xi(\gamma)-\int_{\Gamma_I}\varphi(\gamma(s))\,d\boldsymbol\xi(\gamma)\right|
\le&\int_{\Gamma_I}|\varphi(\gamma(t))-\varphi(\gamma(s))|\,d\boldsymbol\xi(\gamma)\le \|\nabla\varphi\|_{\infty}\\
\le&\int_{t}^s \int_{\Gamma_I}|\dot\gamma|(\tau)\,d\boldsymbol\xi(\gamma)\,d\tau,
\end{align*}
and so if $\tau\mapsto \int_{\Gamma_I}|\dot\gamma|(\tau)\,d\boldsymbol\xi(\gamma)$ belongs to $L^1(I)$ then the map $t\mapsto \int_{\mathbb T^d}\varphi(x)\,d\mu_t(x)$
belongs to $AC(I)$. Moreover
\begin{align*}
\dfrac{d}{dt}\int_{\mathbb T^d}\varphi(x)\,d\mu_t(x)=&\dfrac{d}{dt}\int_{\Gamma_I}\varphi(\gamma(t))\,d\boldsymbol\xi(\gamma)=\int_{\Gamma_I}\nabla\varphi(\gamma(t))\cdot\dot\gamma(t)\,d\boldsymbol\xi(\gamma)\\
=&\int_{\Gamma_I}\int_{\hat e_t^{-1}(x)}\langle\nabla\varphi(\gamma(t)),\dot\gamma(t)\rangle\,d\eta_{t,x}(\gamma)\,d\mu_t(x)\\
=&\int_{\Gamma_I}\langle \nabla\varphi(x),\int_{\hat e_t^{-1}(x)}\dot \gamma(t)\,d\eta_{t,x}(\gamma)\rangle\,d\mu_t(x),
\end{align*}
where we disintegrate $\boldsymbol\xi=\mu_t\otimes\eta_{t,x}$ w.r.t. $\hat e_t$.
\par\medskip\par 
In particular, set \[v_t(x)=\int_{\hat e_t^{-1}(x)}\dot\gamma(t)\,d\eta_{t,x}(\gamma)\] for $\mu_t$-a.e. $x\in\mathbb T^d$ and $\mathscr L^1$-a.e. $t\in I$, it holds
\[\dfrac{d}{dt}\int_{\mathbb T^d}\varphi(x)\,d\mu_t(x)=\int_{\mathbb T^d}\nabla\varphi(x)v_t(x)\,d\mu_t(x),\]
and by the absolute continuity of $t\mapsto \int_{\mathbb T^d}\varphi(x)\,d\mu_t(x)$ this relation holds also in the sense of distributions.
So $\boldsymbol\mu=\{\mu_t\}_{t\in I}$ solves the continuity equation driven by $v=v_t(x)$. Moreover
\begin{align*}
\int_I\|v_t\|_{L^1_{\mu_t}}\,dt=&\int_I\int_{\mathbb T^d}|v_t(x)|\,d\mu_t(x)\,dt\le \int_I\int_{\mathbb T^d}\int_{\hat e_t^{-1}(x)}|\dot\gamma(t)|\,d\eta_{t,x}(\gamma)\,d\mu_t\,dt\\
\le&\iint_{I\times\Gamma_I}|\dot\gamma(t)|\,d\boldsymbol\xi(\gamma)\,dt<+\infty,
\end{align*}
so \eqref{eq:intcond} is satisfied.
\end{proof}

In \cite{JMQ21} the authors used variations generated with optimal displacements to define super/sub differentials, 
such kind of variations can be viewed either as admissible variations generated by measures on curves or as flat variations.

\section{Differentiation in the Wasserstein space}\label{sec:diff}

\subsection{Derivative along admissible variations}

The class of admissible variations can be used to define a concept of derivative as follows.

\begin{definition}\label{def:deriv}
Given a function $U:\mathscr P(\mathbb T^d)\to\mathbb R$, and a family of $q$-admissible variations $\boldsymbol\pi:=\{\pi_t\}_{t\in [0,1]}$ from $\mu$,
we say that $p_\mu^{\boldsymbol\pi}\in L^p_\mu(\mathbb T^d)$, $p>1$, is a \emph{derivative of $U$ at $\mu$ along the family of $q$-admissible variations $\boldsymbol\pi$} if
\[\lim_{t\to 0^+} \dfrac{1}{t}\left|U(\mathrm{pr}_2\sharp\pi_t)-U(\mu)-\iint_{\mathbb T^d\times\mathbb T^d}\langle p_\mu^{\boldsymbol\pi}(x),y-x\rangle\,d\pi_t(x,y)\right|=0.\]
\end{definition}

We show a simple example of computation of the derivative along admissible variations.

\begin{example}
Let $\varphi\in C^1_c(\mathbb T^d)$, and define $U:\mathscr P(\mathbb T^d)\to\mathbb R$ to be $\displaystyle U(\mu)=\int_{\mathbb T^d}\varphi(x)\,d\mu(x)$.
Then $\nabla\varphi$ is a derivative of $U(\cdot)$ at $\mu$ for any $\mu\in\mathscr P(\mathbb R^d)$ along every family $\boldsymbol\pi:=\{\pi_t\}_{t\in [0,1]}$ of $q$-admissible variations starting from $\mu$.
\end{example}
\begin{proof}
Given any $\varphi\in C^1_c(\mathbb T^d)$, and set
\[g_{\varphi}(x,y):=\begin{cases}\dfrac{\varphi(y)-\varphi(x)-\langle\nabla\varphi(x),y-x\rangle}{|y-x|},&\textrm{ for $x\ne y$,}\\\\ 0,&\textrm{ for $x=y$},\end{cases}\]
we have that $g_{\varphi}$ is continuous, with compact support.
H\"older's inequality yields
\begin{align*}
\left|\dfrac{1}{t}\iint_{\mathbb T^d\times\mathbb T^d} \left[\varphi(y)-\varphi(x)-\langle\nabla\varphi(x),y-x\rangle\right]\,d\pi_t(x,y)\right|\hspace{-0.25\textwidth}&\\
=&\dfrac1t\iint_{\mathbb T^d\times\mathbb T^d} |g_\varphi(x,y)|\cdot |y-x|\,d\pi_{t}(x,y)\\
\le&\left(\iint_{\mathbb T^d\times\mathbb T^d} |g_\varphi(x,y)|^p\,d\pi_{t}(x,y)\right)^{1/p}\cdot \dfrac{1}{t}W^{\pi_{t}}_q(\mu,\mathrm{pr}_2\sharp\pi_t).
\end{align*}
By taking the limit along any sequence $t_n\to 0^+$, and recalling the assumptions on $\pi_t$, the continuity of $g_\varphi$ and the fact that $\pi_{t_n}\rightharpoonup(\mathrm{Id},\mathrm{Id})\sharp\mu$,
we obtain
\begin{align*}
\lim_{n\to +\infty}\left|\dfrac{1}{t_n}\iint_{\mathbb T^d\times\mathbb T^d} \left[\varphi(y)-\varphi(x)-\langle\nabla\varphi(x),y-x\rangle\right]\,d\pi_{t_n}(x,y)\right|=0.
\end{align*}
We conclude that $\nabla\varphi$ is a directional derivative by the arbitrariness of the sequence $t_n\to 0^+$.
\end{proof}

%

\begin{corollary}\label{cor:nonsmooth-case}
Suppose that $p_\mu^{\boldsymbol\pi}\in\overline{\{\nabla \varphi:\, \varphi\in C^1_c(\mathbb T^d)\}}^{L^p_\mu}$ is a  derivative of $U$ w.r.t. the family of $q$-admissible variations $\boldsymbol\pi=\{\pi_t\}_{t\in [0,1]}$.
Then for every $\varepsilon>0$ there exists $P^{\boldsymbol\pi}_{\mu}\in C^1_c(\mathbb T^d)$, possibly depending by $\varepsilon$, such that
\[\lim_{t\to 0^+} \dfrac{1}{t}\left|U(\mu_t)-U(\mu)-\int_{\mathbb T^d}P^{\boldsymbol\pi}_{\mu}(z)d(\mu_t-\mu)(z)\right|\le\varepsilon,\]
where $\mu_t=\mathrm{pr}_2\sharp\pi_t$.
In particular, if $p_\mu^{\boldsymbol\pi}=\nabla \hat P^{\boldsymbol\pi}_\mu$ with $\hat P^{\boldsymbol\pi}_\mu\in C^1_c(\mathbb T^d)$, we have
\[\lim_{t\to 0^+} \dfrac{1}{t}\left|U(\mu_t)-U(\mu)-\int_{\mathbb T^d}\hat P^{\boldsymbol\pi}_{\mu}(z)d(\mu_t-\mu)(z)\right|=0.\]
\end{corollary}
\begin{proof}
Let $\varepsilon>0$ and $P^{\boldsymbol\pi}_{\mu}\in C^1_c(\mathbb T^d)$ be such that $\|p_\mu-\nabla P^{\boldsymbol\pi}_\mu\|_{L^p}\le \varepsilon/(C+1)$,
where $C>0$ is the constant appearing in Definition \ref{def:adm-var} (c.).
By H\"older's inequality
\begin{align*}
\lim_{t\to 0^+}&\left|\dfrac{1}{t}\iint_{\mathbb T^d\times\mathbb T^d} \left[P^{\boldsymbol\pi}_\mu(y)-P^{\boldsymbol\pi}_\mu(x)-\langle p_\mu^{\boldsymbol\pi}(x),y-x\rangle\right]\,d\pi_t(x,y)\right|\\
=&\lim_{t\to 0^+}\left|\dfrac{1}{t}\iint_{\mathbb T^d\times\mathbb T^d} \left[P^{\boldsymbol\pi}_\mu(y)-P^{\boldsymbol\pi}_\mu(x)-\langle \nabla P^{\boldsymbol\pi}_\mu(x),y-x\rangle\right]\,d\pi_t(x,y)\right.\\
&\left.-\iint_{\mathbb T^d\times\mathbb T^d}\langle \nabla p_\mu(x)-P^{\boldsymbol\pi}_\mu(x),y-x\rangle\,d\pi_t(x,y)\right|\\
\le&\lim_{t\to 0^+}\left|\dfrac{1}{t}\iint_{\mathbb T^d\times\mathbb T^d} \left[P^{\boldsymbol\pi}_\mu(y)-P^{\boldsymbol\pi}_\mu(x)-\langle \nabla P^{\boldsymbol\pi}_\mu(x),y-x\rangle\right]\,d\pi_t(x,y)\right|+\dfrac{\varepsilon}{C+1} \cdot\dfrac{1}{t}W^{\pi_t}_q(\mu,\mathrm{pr}_2\sharp \pi_t)\\
\le&\varepsilon.
\end{align*}
Thus 
\[\lim_{t\to 0^+} \dfrac{1}{t}\left|U(\mu_t)-U(\mu)-\iint_{\mathbb T^d\times\mathbb T^d}[P^{\boldsymbol\pi}_\mu(y)-P^{\boldsymbol\pi}_\mu(x)]\,d\pi_t(x,y)\right|\le\varepsilon,\]
and we notice that
\[\iint_{\mathbb T^d\times\mathbb T^d}[P^{\boldsymbol\pi}_\mu(y)-P^{\boldsymbol\pi}_\mu(x)]\,d\pi_t(x,y)=\iint_{\mathbb T^d}P^{\boldsymbol\pi}_\mu(z)\,d(\mu_t-\mu)(z).\]
The last assertion follows from taking $P^{\boldsymbol\pi}_\mu=\hat P^{\boldsymbol\pi}_\mu$ for every $\varepsilon>0$. The proof is complete.
\end{proof}

For specific classes of variations, it is possible to provide easy expressions for the computation of the derivative.

\begin{corollary}\label{cor:diffcase}
Let $p_\mu^{\boldsymbol\pi}$ be a derivative of $U$ w.r.t. the family of admissible variations $\boldsymbol\pi=\{\pi_t\}_{t\in [0,1]}$ starting from $\mu$,
and set $\mu_t=\mathrm{pr}_2\sharp\pi_t$.
Then
\begin{enumerate}
\item if $\pi_t=(\mathrm{Id},(\mathrm{Id}+t\phi))\sharp\mu$ is a transport map variation we have 
\[\lim_{t\to 0^+} \dfrac{U(\mu_t)-U(\mu)}{t}=\int_{\mathbb T^d}\langle p^{\boldsymbol\pi}_\mu(x),\phi(x)\rangle\,d\mu(x).\]
\item if $\pi_t=(e_0,e_t)\sharp\boldsymbol\eta$ is a variation along a solution of the continuity equation supported on solutions
of a bounded and continuous vector field $v=v_t(x)$ we have 
\begin{align*}
\lim_{t\to 0^+} \dfrac{U(\mu_t)-U(\mu)}{t}=&\int_{\mathbb T^d}\langle p^{\boldsymbol\pi}_\mu(x),v_0(x)\rangle\,d\mu(x).
\end{align*}
\item if $\pi_t=(\mathrm{Id},\mathrm{Id})\sharp\mu+t[\pi-(\mathrm{Id},\mathrm{Id})\sharp\mu]$ is a flat variation with $\pi\in\mathscr P(\mathbb R^d\times\mathbb R^d)$, $\mathrm{pr}_1\sharp\pi=\mu$, set $\mathrm{pr}_2\sharp\pi=\mu'$ and recalling that $\mathrm{pr}_2\sharp\pi_t=\mu+t(\mu'-\mu)$, we have 
\[\lim_{t\to 0^+} \dfrac{U(\mu_t)-U(\mu)}{t}=\iint_{\mathbb T^d\times\mathbb T^d}\langle p^{\boldsymbol\pi}_\mu(x),y-x\rangle\,d\pi(x,y).\]
\end{enumerate}
\end{corollary}
\begin{proof}\leavevmode\par
1. In this case, we have 
\[\dfrac{1}{t}\iint_{\mathbb T^d\times\mathbb T^d}\langle p^{\boldsymbol\pi}_\mu(x),y-x\rangle\,d\pi_t(x,y)=\int_{\mathbb T^d}\langle p^{\boldsymbol\pi}_\mu(x),\phi(x)\rangle\,d\mu,\]
and the result follows from Definition \ref{def:adm-var}.
\par\medskip\par
2. In this case we have
\begin{align*}
\lim_{t\to 0^+}\dfrac{1}{t}\iint_{\mathbb T^d\times\mathbb T^d}\langle p^{\boldsymbol\pi}_\mu(x),y-x\rangle\,d\pi_t(x,y)
=&\lim_{t\to 0^+}\iint_{\mathbb T^d\times\mathbb T^d}\langle p^{\boldsymbol\pi}_\mu(x),\dfrac{\gamma(t)-\gamma(0)}{t}\rangle\,d\boldsymbol\eta(x,\gamma)\\
=&\lim_{t\to 0^+}\iint_{\mathbb T^d\times\mathbb T^d}\langle p^{\boldsymbol\pi}_\mu(x),\dfrac{1}{t}\int_0^t \dot\gamma(s)\,ds\rangle\,d\boldsymbol\eta(x,\gamma)\\
=&\lim_{t\to 0^+}\iint_{\mathbb T^d\times\mathbb T^d}\langle p^{\boldsymbol\pi}_\mu(x),\dfrac{1}{t}\int_0^t v_s(\gamma(s))\,ds\rangle\,d\boldsymbol\eta(x,\gamma)\\
=&\int_{\mathbb T^d}\langle p^{\boldsymbol\pi}_\mu(x),v_0(x)\rangle\,d\mu(x),
\end{align*}
and the result follows from Definition \ref{def:adm-var}.
\par\medskip\par
3. In this case we have 
\[\dfrac{1}{t}\iint_{\mathbb T^d\times\mathbb T^d}\langle p^{\boldsymbol\pi}_\mu(x),y-x\rangle\,d\pi_t(x,y)=\iint_{\mathbb T^d\times\mathbb T^d}\langle p^{\boldsymbol\pi}_\mu(x),y-x\rangle\,d\pi(x,y),\]
and the result follows from Definition \ref{def:adm-var}.
\end{proof}

\begin{remark}
Notice that combining Corollary \ref{cor:diffcase} (3) and Corollary \ref{cor:nonsmooth-case}, if $\boldsymbol\pi=\{\pi_t\}$ is a family of flat variations starting from $\mu$
and there exists $P_\mu^{\boldsymbol\pi}\in C^1_c(\mathbb T^d)$ such that $p^{\boldsymbol\pi}_\mu:=\nabla P_\mu^{\boldsymbol\pi}$ is a derivative for $U$ w.r.t. $\boldsymbol\pi$,
then 
\[\lim_{t\to 0^+} \dfrac{U(\mu_t)-U(\mu)}{t}=\lim_{t\to 0}\dfrac{1}{t}\iint_{\mathbb T^d\times\mathbb T^d}\left[P^{\boldsymbol\pi}_{\mu}(y)-P^{\boldsymbol\pi}_{\mu}(x)\right]\,d\pi_t(x,y)=\int_{\mathbb T^d} P^{\boldsymbol\pi}_\mu(z)\,d(\mu'-\mu)(z),\]
where $\mu'=\mathrm{pr}_2\sharp\pi$.
\end{remark}

\begin{remark}
Suppose that $p_\mu\in L^p(\mathbb T^d)$ is a  derivative for a family of  flat variations starting from $\mu$ and suppose moreover that  $(\mu,x)\mapsto p_ \mu $ is continuous, then  the integral
\[\dfrac{1}{t}\int_{\mathbb T^d\times\mathbb T^d}\langle p^{\boldsymbol\pi}_\mu(x),y-x\rangle\,d\pi_t(x,y)\]
does actually not depend on $t$.

We have the same result if $ p_ \mu$ is a derivative for family of variations generated by a transport map.
\end{remark}

\begin{proposition}[Integral form]\label{prop:infbeh}
	Let $U:\mathscr P(\mathbb T^d)\to \mathbb R$ be $W_p$-Lipschitz continuous. Suppose that 
	\begin{enumerate}
		\item[]
		\item $U(\cdot)$ is differentiable w.r.t. every family of transport map variations,
		\item[]
		\item for any $\mu\in\mathbb R^d$ the derivative $p^{\boldsymbol\pi}_{\mu}(\cdot)$ of $U$ w.r.t. to every family of transport map variations $\boldsymbol\pi$ 
		from $\mu$ depends on $\mu$ only (and not by $\boldsymbol\pi$), i.e., there exists a Borel map $(\mu,x)\mapsto p_{\mu}(x)$ with $p_\mu(\cdot)\in L^q_{\mu}(\mathbb T^d)$ 
		such that $p_\mu(\cdot)=p^{\boldsymbol\pi}_{\mu}(\cdot)$ for every family of transport map variations $\boldsymbol\pi$ from $\mu$.
	\end{enumerate}
	Then for every $W_p$-absolutely continuous curve $\boldsymbol\mu=\{\mu_\tau\}_{\tau\in I}$ and $s,t\in I$, we have 
	\[U(\mu_t)-U(\mu_s)=\int_s^t \int_{\mathbb T^d}\langle p_{\mu_\tau}(x),v_\tau(x)\rangle\,d\mu_\tau(x)\,d\tau,\]
	where $v=v_t(x)$ is characterized to be the unique Borel vector field such that 
	\begin{itemize}
		\item the continuity equation \eqref{eq:conteq} holds,
		\item for a.e. $t\in I$ we have that $v_t\in L^p_{\mu_t}(\mathbb T^d)$ with $\|v_t\|_{L^p_{\mu_t}}\le |\dot\mu_t|$, where $t\mapsto |\dot\mu_t|$ is the metric derivative of $\boldsymbol\mu$ (in particular, we have also that \eqref{eq:intcond} is satisfied).
	\end{itemize}
\end{proposition}
\begin{proof}
	The map $g:I\to\mathbb R$ defined by $g(s):=U(\mu_s)$ is absolutely continuous, indeed
	\[|U(\mu_t)-U(\mu_s)|\le \mathrm{Lip}(U)\cdot W_p(\mu_t,\mu_s) \le \int_s^t \mathrm{Lip}(U) |\dot\mu_{\tau}|\,d\tau,\]
	where $\tau\mapsto |\dot\mu_{\tau}|$ denotes the metric derivative of $\boldsymbol\mu$.
	In particular,
	\[U(\mu_t)-U(\mu_s)=\int_s^t \dot g(\tau)\,d\tau,\]
	where $\dot g(\tau)$ is defined for a.e. $t\in I$.
	To conclude the proof we must prove that 
	\[\displaystyle\dot g(s)=\int_{\mathbb T^d}\langle p_{\mu_s}(x),v_s(x)\rangle\,d\mu_s(x).\]
	Given $\tau\in I$ and $h>0$ such that $\tau+h\in I$, we define $\hat \mu_{\tau,h}=(\mathrm{Id}+hv_\tau)\sharp\mu_\tau$.
	We have 
	\begin{align*}
		\left|\dfrac{g(\tau+h)-g(\tau)}{h}-\int_{\mathbb T^d}\langle p_{\mu_\tau}(x),v_\tau(x)\rangle\,d\mu_\tau(x)\right|\hspace{-6cm}&\\
		\le&\left|\dfrac{U(\hat \mu_{\tau,h})-U(\mu_\tau)}{h}-\int_{\mathbb T^d}\langle p_{\mu_\tau}(x),v_\tau(x)\rangle\,d\mu_\tau(x)\right|+\left|\dfrac{U(\mu_{\tau+h})-U(\hat \mu_{\tau,h})}{h}\right|\\
		\le&\left|\dfrac{U(\hat \mu_{\tau,h})-U(\mu_\tau)}{h}-\int_{\mathbb T^d}\langle p_{\mu_\tau}(x),v_\tau(x)\rangle\,d\mu_\tau(x)\right|+\mathrm{Lip}(U)\cdot \dfrac{W_p(\mu_{\tau+h},\hat \mu_{\tau,h})}{h}.
	\end{align*}
	By letting $h\to 0$, the second term vanishes for a.e. $\tau\in I$ thanks to Proposition 8.4.6 in \cite{AGS}, and the first vanishes by Corollary \ref{cor:diffcase}, so $\dot g(\tau)$ has the desired expression.
\end{proof}

\subsection{Comparison with other concept of differentials}

We recall now some other concept of differentials in Wasserstein space, and extensively used in particular  for multiagent systems of in the Mean Field Game theory framework.
Our aim is to compare these concept with the ones defined in the previous section in the smooth case.

\begin{definition}[Definition 2.2.1 and Definition 2.2.2 in \cite {CDLL}]\label{def:CDLLder}
Let $U:\mathscr P(\mathbb T^d)\to\mathbb R$ be a function.
\begin{itemize}
\item We say that $U(\cdot)$ is $C^1$ if there exists a continuous map $\dfrac{\delta U}{\delta m}:\mathscr P(\mathbb T^d)\times\mathbb T^d\to\mathbb R$ 
such that for any $m,m'\in\mathscr P(\mathbb T^d)$ it holds
\[\lim_{s\to 0^+}\dfrac{U((1-s)m+sm')-U(m)}{s}=\int_{\mathbb T^d} \dfrac{\delta U}{\delta m}(m,y)\,d(m'-m)(y).\]
Note that $\dfrac{\delta U}{\delta m}$ is defined up to an additive constant. We adopt the normalization convention
\[\int_{\mathbb T^d} \dfrac{\delta U}{\delta m}(m,y)\,dm(y)=0.\]
\item Given a $C^1$ map $U(\cdot)$, if $\dfrac{\delta U}{\delta m}$ is of class $C^1$ with respect to the second variable, the instrinsic derivative 
$D_m U:\mathscr P(\mathbb T^d)\times\mathbb T^d\to\mathbb R^d$ is defined by
\[D_m U(m,y)=D_y\dfrac{\delta U}{\delta m}(m,y).\]
\end{itemize}
\end{definition}
Observe that in those definitions a joint regularity is requested, \emph{both} w.r.t. the dependence on the measure variable and to the finite dimensional one.

\begin{remark}
	We recall that, according to the statement followinf Formula 2.2 p. 29 and before Definition 2.2.2 of \cite{CDLL}, functions that are $C^1$ in the sense of \cite{CDLL} are automatically $W_1$-Lipschitz continuous.
\end{remark}

The next result shows that, in the regular case, these definitions of derivative almost coincide.

\begin{theorem} \label{thm:main}
	Let $U:\mathscr P(\mathbb T^d)\to\mathbb R$ 	be Lipschitz continuous.
	\begin{enumerate}
		\item[A)] Assume that $U$ is $C^1$ in the sense of Definition \ref{def:CDLLder}, and denote by $P_m(y) :=\dfrac{\delta U}{\delta m}(m,y)$ and $p_m(y):= \nabla P _m(y) = D_m U(m,y)$. 
		Suppose that $(m,y)\mapsto p_m(y)$ is continuous with respect to both variables $y$ and $m$. Then $U$ is derivable along every family of admissible variations and its derivative is $p_m$.
		In particular, 
		\begin{enumerate}
			\item[A-1)] $U$ is derivable along all families of flat variations and its derivative is $p_m$.
			\item[A-2)] $U$ is derivable along all families of transport map variations  generated by bounded Borel maps and its derivative is $p_m$.
			\item[A-3)] $U$ is derivable along all families of Eulerian variations and its derivative is $p_m$.
		\end{enumerate}
            \item[B)] For any $m\in\mathscr P(\mathbb T^d)$, let $P_m\in C^1(\mathbb T^d; \mathbb R)$ and set $p_m(y):= \nabla P _m(y)\in C^0(\mathbb T^d;\mathbb R^d)$. 
		We assume that $(y,m) \mapsto p_m(y) $ is continuous with respect to both variables and one between these conditions
		\begin{enumerate}
			\item[B-1)] $U$ is derivable along all families of flat variations, and its derivative at $m$ is $p_m(\cdot)$,
			\item[B-2)] $U$ is derivable along all families of transport map variations, and its derivative at $m$ is $p_m(\cdot)$,
			\item[B-3)] $U$ is derivable along all families of Eulerian variations, and its derivative at $m$ is $p_m(\cdot)$.
		\end{enumerate}
		Then $U$ is $C^1$ in the sense of Definition \ref{def:CDLLder} and  $\displaystyle\dfrac{\delta U}{\delta m}(m,y) =P_m(y)-\int_{\mathbb T^d}P_m(z)\,dm(z)$. 
	\end{enumerate}
\end{theorem}
\begin{proof} 
	In order to prove part A), we assume that $U$ is $C^1$ with derivative $P_m$. According to Formula 2.2 p. 29 in \cite{CDLL}, in this case it holds
	\begin{align}\label{eq:restint}  
		U(m')-U(m)=\int_0^1\int_{\mathbb T^d}\dfrac{\delta U}{\delta m}((1-\tau)m+\tau m',y)\,d(m'-m)(y)\,d\tau,&&\forall m,m'\in\mathscr P(\mathbb T^d).
	\end{align} 
	Let now $m\in \mathscr P(\mathbb T^d)$ be fixed and $\boldsymbol\pi=\{\pi_s\}_{s\in [0,1]}$ be an admissible variation from $m$.
	For any $\tau\in [0,1]$ and $\psi\in C^0([0,1];[0,1])$ with $\psi(0)=0$ and $\psi(1)=1$ define 
	\begin{align*}
		\pi_{\tau,s}=(\mathrm{Id},\mathrm{Id})\sharp m+\psi(\tau)(\pi_s-(\mathrm{Id},\mathrm{Id})\sharp m),&&m_{\tau,s}=\mathrm{pr}_2\sharp\pi_{\tau,s},&&m_s=\mathrm{pr}_2\sharp\pi_s,
	\end{align*}
	Notice that $m_{\tau,s}=m+\psi(\tau)(m_s-m)$ and so $m_s=m_{1,s}$.
	\par\medskip\par 
	For every integrable function $k:I\times I\times\mathbb T^d\times\mathbb T^d\to\mathbb R$ such that $k(\tau,s,x,x)=0$ it holds
	\begin{equation}\label{eq:relsymmpart}
		\iint_{\mathbb T^d\times\mathbb T^d}k(\tau,s,y,x)\,d\pi_{\tau,s}(x,y)=\psi(\tau)\iint_{\mathbb T^d\times\mathbb T^d}k(\tau,s,y,x)\,d\pi_{s}(x,y)
	\end{equation}
	Define the continuous function $g:[0,1]\times [0,1]\times\mathbb T^d\times\mathbb T^d\to\mathbb R$ by
	\[g(\tau,s,y,x):=P_{m_{\tau,s}}(y)-P_{m_{\tau,s}}(x)-\langle p_{m_{\tau,s}}(x),y-x\rangle,\]
	and notice that $g(\tau,s,x,x)=0$.
	\par\medskip\par
	By the Mean Value Theorem, there exists $\lambda\in [0,1]$ such that 
	\begin{align*}
		g(\tau,s,y,x)=&g(\tau,s,y,x)-g(\tau,s,x,x)=\langle \partial_y g(\tau,s,x+\lambda(y-x),x),y-x\rangle\\
		=&\langle p_{m_{\tau,s}}(x+\lambda(y-x))-p_{m_{\tau,s}}(x),y-x\rangle.
	\end{align*}
	Since $(m,z)\mapsto p_{m}(z)$ is continuous on the compact $\mathscr P(\mathbb T^d)\times \mathbb T^d$, it is uniformly continuous,
	and in particular there exists a continuous incresing function $\omega:[0,+\infty[\to [0,+\infty[$ such that $\omega(0)=0$ and
	\begin{align*}
		|p_{m_{\tau_1,s_1}}(z_1)-p_{m_{\tau_2,s_2}}(z_2)|\le \omega(W_p(m_{\tau_1,s_1},m_{\tau_2,s_2})+|z_1-z_2|),
	\end{align*}
	for all $(\tau_i,s_i,z_i)\in [0,1]\times[0,1]\times \mathbb T^d$, $i=1,2$.
	\par\medskip\par
	We thus obtain 
	\[|g(\tau,s,y,x)|\le |y-x|\cdot \omega(\lambda|y-x|)\le |y-x|\cdot \omega(|y-x|).\]
	\par\medskip\par
	According to \eqref{eq:restint}, we have
	\begin{align*}
		\dfrac 1s\cdot\left[U(m_s)-U(m)-\iint_{\mathbb T^d \times \mathbb T^d}\langle p_m(x), y -x \rangle\, d\pi_s (x,y)\right]\hspace{-8cm}&\\
		=&\dfrac 1s\int_0^1\int_{\mathbb T^d}P_{m_{\tau,s}}(y)\,d(m_{\tau,s}-m)(y)\,dt -\dfrac 1s \iint_{\mathbb T^d \times \mathbb T^d} \langle p_m(x), y -x \rangle\,d\pi_s (x,y)\\
		=&\dfrac 1s\int_0^1\iint_{\mathbb T^d \times \mathbb T^d}\left[P_{m_{\tau,s}}(y) -P_{m_{\tau,s}}(x)\right]\,d\pi_{\tau,s}(x,y)\,d\tau - \dfrac 1s\iint_{\mathbb T^d \times \mathbb T^d}\langle p_m(x), y -x \rangle\, d\pi_s (x,y)\\
		=&\dfrac 1s\int_0^1\iint_{\mathbb T^d \times \mathbb T^d}g(\tau,s,y,x)\,d\pi_{\tau,s}(x,y)\,d\tau+\\
		&+\dfrac 1s\int_0^1\iint_{\mathbb T^d \times \mathbb T^d}\langle p_{m_{\tau,s}}(x),y-x\rangle\,d\pi_{\tau,s}(x,y)\,d\tau-\dfrac 1s\iint_{\mathbb T^d \times \mathbb T^d}\langle p_m(x), y -x \rangle\, d\pi_s (x,y)\\
	\end{align*}
	The behaviour as $s\to 0^+$ of the first term can be estimated by H\"older's inequality, by \eqref{eq:relsymmpart} since $g(\tau,s,x,x)=0$, and using the fact that $|\psi(\tau)|\le 1$, as
	\begin{align*}
		A(s):=&\left|\dfrac 1s\int_0^1\iint_{\mathbb T^d \times \mathbb T^d}g(\tau,s,y,x)\,d\pi_{\tau,s}(x,y)\,d\tau\right|
		=\left|\dfrac1s\int_0^1\iint_{\mathbb T^d \times \mathbb T^d}\psi(\tau)g(\tau,s,y,x)\,d\pi_{s}(x,y)\,d\tau\right|\\
		\le&\dfrac{1}{s}\int_0^1 \iint_{\mathbb T^d \times \mathbb T^d}\psi(\tau)\omega(|x-y|)\cdot |x-y|\,d\pi_{s}(x,y)\,d\tau\\
		\le&\dfrac{W_q^{\boldsymbol\pi}(m,m_s)}{s}\cdot  \left(\iint_{\mathbb T^d \times \mathbb T^d}\omega^p(|x-y|)\,d\pi_{s}(x,y)\right)^{1/p}.
	\end{align*}
	Recalling the definition of admissible variation (Definition \ref{def:adm-var}), as $s\to 0^+$,
	the term $\dfrac{W_q^{\boldsymbol\pi}(m,m_s)}{s}$ is uniformly bounded, and by Dominated Convergence Theorem
	\[\lim_{s\to 0^+}\iint_{\mathbb T^d \times \mathbb T^d}\omega^p(|x-y|)\,d\pi_{s}(x,y)=0,\]
	since $(x,y)\mapsto \omega^p(|x-y|)$ is bounded and continuous and $\pi_s\rightharpoonup (\mathrm{Id},\mathrm{Id})\sharp m$. Thus $A(s)\to 0$.
	\par\medskip\par
	Concerning the other terms, with similar estimates we have
	\begin{align*}
		B(s):=&\left|\dfrac 1s\int_0^1\iint_{\mathbb T^d \times \mathbb T^d}\langle p_{m_{\tau,s}}(x),y-x\rangle\,d\pi_{\tau,s}(x,y)\,d\tau-\dfrac 1s\iint_{\mathbb T^d \times \mathbb T^d}\langle p_m(x), y -x \rangle\, d\pi_s (x,y)\right|\\
		=&\left|\dfrac 1s\int_0^1\psi(\tau)\iint_{\mathbb T^d \times \mathbb T^d}\langle p_{m_{\tau,s}}(x),y-x\rangle\,d\pi_{s}(x,y)\,d\tau-\dfrac 1s\iint_{\mathbb T^d \times \mathbb T^d}\langle p_m(x), y -x \rangle\, d\pi_s (x,y)\right|\\
		=&\left|\dfrac 1s\int_0^1\iint_{\mathbb T^d \times \mathbb T^d}\langle \psi(\tau) p_{m_{\tau,s}}(x)-p_m(x),y-x\rangle\,d\pi_{s}(x,y)\,d\tau\right|\\
		\le&\dfrac 1s\iint_{\mathbb T^d \times \mathbb T^d}\left|\int_0^1[\psi(\tau) p_{m_{\tau,s}}(x)\,d\tau-p_m(x)]\,d\tau\right|\cdot|y-x|\,d\pi_{s}(x,y)\\
		\le&\dfrac{W_q^{\boldsymbol\pi}(m,m_s)}{s}\cdot \left(\int_{\mathbb T^d}\left|\int_0^1[\psi(\tau) p_{m_{\tau,s}}(x)-p_m(x)]\,d\tau\right|^p\,dm(x)\right)^{1/p}
	\end{align*}
	As before, the first term is uniformly bounded as $s\to 0^+$. Concerning the other, we notice that $m_{\tau,s}\to m$ for every $\tau$ as $s\to 0^+$. By Dominated Convergence Theorem
	we conclude that 
	\begin{align*}
		\limsup_{s\to 0^+}B(s)\le&C\cdot \int_0^1 |\psi(\tau)-1|\,d\tau\cdot \|p_m\|_{L^p_m},
	\end{align*}
	where $C$ is the constant appearing in item c. of Definition \ref{def:adm-var}.
	\par\medskip\par  
	Thus we proved that for any $\psi\in C^0([0,1];[0,1])$ with $\psi(0)=0$ and $\psi(1)=1$ it holds
	\begin{align*}
		\limsup_{s\to 0^+}\dfrac 1s&\cdot\left|U(m_s)-U(m)-\iint_{\mathbb T^d \times \mathbb T^d}\langle p_m(x), y -x \rangle\, d\pi_s (x,y)\right|\le\limsup_{s\to 0^+} A(s)+B(s)\\
		\le&C\cdot \int_0^1 |\psi(\tau)-1|\,d\tau\cdot \|p_m\|_{L^p_m},
	\end{align*}
	and by passing to the infimum on $\psi\in C^0([0,1];[0,1])$ with $\psi(0)=0$ and $\psi(1)=1$ we have that the last expression vanishes,
	hence $p_m(\cdot)$ is a derivative of $U(\cdot)$ at $m$ w.r.t. $\boldsymbol\pi$.
	\par\medskip\par
	In order to prove part B), fix $m\in\mathscr P(\mathbb T^d)$ and let $p_m(y):= \nabla P _m(y)$ be as in the statement of the part B) of the statement.
	\par\medskip\par
	Notice that the assumptions B-3) implies both B-2) and B-1). Thus it is sufficient to prove B-1) and B-2).
	\par\medskip\par
	We prove B-1). Fix $\pi\in\Pi(m,m')$, and define the family of flat variations $\boldsymbol\pi=\{\pi_s\}_{x\in [0,1]}$ with 
	$\pi_s := (1-s)(\mathrm{Id},\mathrm{Id})\sharp m +s\pi$. Set $m_s:=\mathrm{pr}_2\sharp\pi_s=(1-s)m+sm'$. 
	By assumption, $p_m$ is the derivative also with respect with this family, and so by Corollary \ref{cor:nonsmooth-case} we have
	\begin{align*}
		0=&\lim_{s\to 0^+}\dfrac{1}{s}\left[U(m_s)-U(m)-\int_{\mathbb T^d} P_m(z)\,d(m_s-m)\right]\\
		=&\lim_{s\to 0^+}\left[\dfrac{U(m_s)-U(m)}{s}-\int_{\mathbb T^d} P_m(z)\,d(m'-m)\right],
	\end{align*}
	so 
	\[\lim_{s\to 0^+}\dfrac{U(m_s)-U(m)}{s}=\int_{\mathbb T^d} P_m(z)\,d(m'-m),\]
	Notice that the same formula holds if we replace $P_m(z)$ with $P_m(z)+k$ for any $k\in\mathbb R$. By choosing $\displaystyle k=-\int_{\mathbb T^d}P_m(z)\,dm(z)$, we have that
	$U$ is $C^1$ in the sense of Definition \ref{def:CDLLder} and $\displaystyle\dfrac{\delta U}{\delta m}(m,y)=P_m(y)-\int_{\mathbb T^d}P_m(z)\,dm(z)$, so the normalization condition
	on $\displaystyle\dfrac{\delta U}{\delta m}(m,y)$ is also satisfied.
	\par\medskip\par
	We prove now B-2). Let $m,m'\in\mathscr P(\mathbb T^d)$ and set $m_s=(1-s)m+sm'$. The curve $s\mapsto m_s$ is clearly absolutely continuous (indeed, $W_1$-Lipschitz continuous), therefore,  
	by Proposition \ref{prop:infbeh}, for all $s,t\in I$ we have 
	\[U(\mu_t)-U(\mu_s)=\int_s^t \int_{\mathbb T^d}\langle p_{\mu_\tau}(x),v_\tau(x)\rangle\,d\mu_\tau(x)\,d\tau,\]
	where $v=v_t(x)$ is characterized to be the unique vector field such that \eqref{eq:conteq} holds, and for a.e. $t\in I$ we have that 
	$v_t\in L^p_{\mu_t}(\mathbb T^d)$ with $\|v_t\|_{L^p_{\mu_t}}\le |\dot\mu_t|$, where $t\mapsto |\dot\mu_t|$ is the metric derivative of $\boldsymbol\mu$.
	Let $\varphi\in C^1$. Since $s\mapsto \int_{\mathbb T^d}\varphi(x)\, dm_s(x)$ is absolutely continuous, for a.e. $s\in I$ we have 
	\[\dfrac{d}{d\tau}\int_{\mathbb T^d}\varphi(x)\, dm_\tau(x)=\int_{\mathbb T^d}\varphi(x) \,d(m-m')=\int_{\mathbb T^d}\langle \nabla\varphi(x),v_\tau(x)\rangle\,d\mu_\tau\]
	Thus for a.e. $t\in I$ and all $\varphi\in C^1$ it holds  
	\[\int_{\mathbb T^d}\varphi(x) \,d(m-m')=\int_{\mathbb T^d}\langle \nabla\varphi(x),v_s(x)\rangle\,d\mu_s.\]
	Thus
	\[U(\mu_t)-U(\mu_s)=\int_s^t \int_{\mathbb T^d} P_{\mu_\tau}(x)\,d(m'-m)(x)\,d\tau,\]
	By taking $s=0$, dividing by $t$, sending $t\to 0^+$ and exploiting the continuity of $(\tau,x)\mapsto P_{\mu_\tau}(x)$, we obtain
	\[\lim_{t\to 0}U(\mu_t)-U(\mu)=\int_{\mathbb T^d} P_{\mu}(x)\,d(m'-m)(x),\]
	as desired.
\end{proof}

\subsection{Super/sub differentials} 
When the function $U$ is not smooth
enough to be $C^1$ in the sense of Definition \ref{def:CDLLder},our concept of variations allows to define super and sub differential.
\begin{definition}[Sub/superdifferential]
Let $U:[0,T]\times \mathbb{T}^d\times \mathscr{P}_2(\mathbb{T}^d) \rightarrow \mathbb{T}^d$, we define the superdifferential of $U$ to be the set 
$\partial^+ U(t_0,x_0,\mu_0)$ of all $(p_{t_0},p_{x_0},p_{\mu_0}(\cdot))$ such that
{\small\begin{equation}
\limsup_{\substack{h\rightarrow 0^+\\ |y|\rightarrow 0}} \dfrac{\displaystyle U(t_0+h,\xi_0+y,\mathrm{pr}_2\sharp\pi_h)-U(t_0,\xi_0,\mu_0)- p_{t_0}h-\langle p_{x_0},y\rangle -\int_{\mathbb{T}^d\times\mathbb{T}^d } \langle p_{\mu_0}(x),\xi-x\rangle \, d \pi_h(x,\xi)}{h+|y|}\leq 0,
\end{equation}}
for every $\{\pi_h\}_{h\in [0,1]}$ family of admissible variations from $\mu_0$.
\par\medskip\par
Similarly, we define the subdifferential of $U$ to be the set $\partial^- U(t_0,x_0,\mu_0)$ of all $(p_{t_0},p_{x_0},p_{\mu_0}(\cdot))\in\mathbb R\times\mathbb R^d\times L^2_{\mu_0}$ such that
{\small\begin{equation}
\liminf_{\substack{h\rightarrow 0^+\\ |y|\rightarrow 0}} \dfrac{\displaystyle U(t_0+h,\xi_0+y,\mathrm{pr}_2\sharp\pi_h)-U(t_0,\xi_0,\mu_0)- p_{t_0}h-\langle p_{x_0},y\rangle -\int_{\mathbb{T}^d\times\mathbb T^d} \langle p_{\mu_0}(x),\xi-x\rangle \, d \pi_h(x,\xi)}{h+|y|}\geq 0.
\end{equation}}
\end{definition}

Observe that the notion of sub/super differential introduced in \cite{CQ} coincide with the notion of sub/super differential defined through all transport map variations.

\begin{remark}
	When admissible variations are limited to locally Lipschitz transport maps variations, our notion of sub/super differential recovers the contingent directional derivative of Definition 2.15 in \cite{BadF}.
\end{remark}

\begin{definition}[First order HJB equation]\label{def:viscosol}
	We consider an equation in the form
	\begin{equation}\label{eq:HJB}
		\partial_t w(t,\xi,\mu) + H (t,\xi,\mu, D_x w, D_\mu w) = 0.
	\end{equation}
	Let $U:[0,T]\times \mathbb{T}^d\times \mathscr{P}_2(\mathbb{T}^d) \rightarrow \mathbb{T}^d$, we say that
	\begin{itemize}
		\item $U$ is a subsolution of  \eqref{eq:HJB} iff for all $(t,x,\mu) \in [0,T[\times \mathbb{T}^d \times \mathscr{P}_2(\mathbb{T}^d)$ and $(p_t,p_x,p_\mu)\in \partial^+ U(t,x,\mu)$ it holds
		\[p_t + H(t,\xi,\mu,p_x,p_\mu) \geq 0.\]
		\item $U$ is a supersolution  of  \eqref{eq:HJB} iff for all $(t,x,\mu) \in [0,T[\times \mathbb{T}^d \times \mathscr{P}_2(\mathbb{T}^d)$ and $(p_t,p_x,p_\mu)\in \partial^- U(t,x,\mu)$ it holds
		\[p_t + H(t,\xi,\mu,p_x,p_\mu) \leq 0.\]
	\end{itemize}
\end{definition}

\section{Model description}\label{sec:model}

We consider a multiagent system where a central planner controls a crowd of agents (the \emph{followers}) and a distinguished agent (the \emph{leader}).
The dynamics of the leader is affected by the current time, the current position of the leader and by the followers' crowd configuration. 
The dynamics of each follower is affected by the current time, by his/her current position, by the position of the leader, and by the the followers' crowd configuration.
Neither the leader nor any follower can single out an agent in the crowd of followers. 
The aim of the central planner is to minimize a global cost of Bolza type in a compact interval of time $I$.
\par\medskip\par

Let $I=[a,b]$ be a compact interval of $\mathbb R$, $k\in L^1(I)$. We set $X=\mathbb T^d$, $Y=C^0(I;X)$, $Z=C^0(I;\mathscr P_2(X))$.

Suppose that
\begin{itemize}
\item[($D_F$)] $F:I\times X\times\mathscr P_2(X)\rightrightarrows X$ is a Borel measurable set valued map,
with nonempty compact convex images, satisfying
\begin{align*}
d_H(F(t,x_1,\theta_1),F(s,x_2,\theta_2))\le k \left(|t-s|+|x_1-x_2|+W_2(\theta_1,\theta_2)\right)
\end{align*}
for all $x_i\in X$, $\theta_i\in\mathscr P_2(X)$, $i=1,2$, $t\in I$;
\item[($D_G$)] $G:I\times X\times X\times\mathscr P_2(X)\rightrightarrows X$ is a Borel mesurable set valued map,
with nonempty compact convex images, satisfying
\begin{align*}
d_H(G(t,x_1,y_1,\theta_1),G(s,x_2,y_2,\theta_2))\le&k \left(|t-s|+|x_1-x_2|+|y_1-y_2|+W_2(\theta_1,\theta_2)\right),\\
\end{align*}
for all $x_i,y_i\in X$, $\theta_i\in\mathscr P_2(X)$, $i=1,2$, $t\in I$.
\end{itemize}

\par\medskip\par
We consider the following differential problem in the unknowns $(\xi(\cdot),\boldsymbol\mu=\{\mu_t\}_{t\in I})\in Y\times Z$
\begin{align}\label{def:diffprob}\begin{cases}
\dot \xi(t)\in F(t,\xi(t),\mu_t),&\textrm{ for a.e. $t\in I$},\\\\
\partial_t \mu_t+\mathrm{div}(v_t\mu_t)=0,&\textrm{ in the sense of distributions},
\end{cases}\end{align}
where the coupling between the two equations is given by assuming that the vector field $v=v_t(y)$ 
can be any Borel selection of the set-valued map $(t,y)\mapsto G(t,\xi(t),y,\mu_t)$. 

\par\medskip\par
Our aim is to prove the following result about the existence of solutions and regularity of the solution map.

\begin{proposition}
Under the above assumptions on $F,G$, define the set valued map 
\[\mathscr A_I:X\times\mathscr P_2(X)\rightrightarrows Y\times Z\]
by setting
\[\mathscr A_I(\bar x,\bar\mu):=\left\{(\xi(\cdot),\boldsymbol\mu)\in Y\times Z:\,t\mapsto(\xi(t),\mu_t)\textrm{ solves }
\eqref{def:diffprob}\textrm{ with }(\xi(a),\mu_a)=(\bar x,\bar\mu)\right\}.\]
Then $\mathscr A_I(\cdot)$ is a Lipschitz continuous set valued map with nonempty compact images.
\end{proposition}
\begin{proof}
The strategy of the proof will be to reformulate the differential problem \eqref{def:diffprob} as a fixed point problem
of a suitable set valued map.
\par\medskip\par
To this aim, for any given $\boldsymbol\theta=\{\theta_t\}_{t\in I}\in Z$ and $w\in Y$, we define
\begin{align*}
S_F^{\boldsymbol\theta,w}(x):=&\left\{\gamma\in AC(I;X):\, \dot \gamma(t)\in F(t,w(t),\theta_t), \textrm{ for a.e. $t\in I$},\,\gamma(0)=x\right\}\subseteq Y,\\
S_G^{\boldsymbol\theta,w}(x):=&\left\{\xi\in AC(I;X):\, \dot \xi(t)\in G(t,w(t),\xi(t),\theta_t), \textrm{ for a.e. $t\in I$},\,\xi(0)=x\right\},\\
\boldsymbol{S_G}^{\boldsymbol\theta,w}(\mu):=&\left\{\boldsymbol\eta\in\mathscr P(X\times Y):\, \boldsymbol\eta(\mathrm{graph}\,S_G^{\boldsymbol\theta,w})=1,\,\mathrm{pr}_1\sharp\boldsymbol\eta=\mu\right\},\\
\boldsymbol{\Upsilon_G}^{\boldsymbol\theta,w}(\mu):=&\left\{\boldsymbol\mu=\{\mu_t\}_{t\in I}\in Z:\, \mu_t=e_t\sharp\boldsymbol\eta, \boldsymbol\eta\in \boldsymbol{S_G}^{\boldsymbol\theta,w}(\mu)\right\},
\end{align*}
where the evaluation map $e_t:X\times Y\to X$ is the $1$-Lipschiz continuous map $e_t(x,\gamma)=\gamma(t)$.
\par\medskip\par
We will prove that 
\[\mathscr A_I(x,\mu)=\left\{(w,\boldsymbol\theta)\in Y\times Z:\,(w,\boldsymbol\theta)\in S_F^{\boldsymbol\theta,w}(x)\times\boldsymbol{\Upsilon_G}^{\boldsymbol\theta,w}(\mu)\right\},\]
i.e., the fixed point of the set valued map $(w,\boldsymbol\theta)\mapsto S_F^{\boldsymbol\theta,w}(x)\times\boldsymbol{\Upsilon_G}^{\boldsymbol\theta,w}(\mu)$.
\par\medskip\par
According to classical Filippov's Theorem for Differential Inclusions (see e.g. Theorem 1 in Chapter 2, Section 4 of \cite{AC}), for fixed $(w,\boldsymbol\theta)\in Y\times Z$
the set valued map $S_F^{\boldsymbol\theta,w}(\cdot)$ and $S_G^{\boldsymbol\theta,w}(\cdot)$ are Lipschitz continuous with compact images (contained in $AC(I;X)$). 
The map $X\times Y\ni(x,\xi)\mapsto\mathrm{dist}\left(\xi, S_G^{\boldsymbol\theta,w}(x)\right)$ is therefore continuous (see e.g. Corollary 1.4.17 in \cite{AuF}). 
\par\medskip\par
By applying Lemma \ref{lemma:projection} with $\mathbb X=X\times Y$, $\mathbb Y=Y$, $h(x,\gamma)=\gamma$, $\mathbb F(x,\gamma)=S_G^{\boldsymbol\theta,w}(x)\subseteq Y$,
there exists a Borel map $\mathrm{Pr}^{\boldsymbol\theta,w}:X\times Y\to Y$ such that $\mathrm{Pr}^{\boldsymbol\theta,w}(x,\gamma)\in S_G^{\boldsymbol\theta,w}(x)$
and 
\[\mathrm{dist}(\gamma,\mathrm{Pr}^{\boldsymbol\theta,w}(x,\gamma))=\mathrm{dist}(\gamma,S_G^{\boldsymbol\theta,w}(x)).\]
We define a Borel map $T^{\boldsymbol\theta,w}:X\times Y\times X\to (X\times Y)\times (X\times Y)$ by setting 
\[T^{\boldsymbol\theta,w}(x_1,\gamma_1,x_2):=((x_1,\gamma_1),\mathrm{Pr}^{\boldsymbol\theta,w}(x_2,\gamma_1)).\]
According to Lemma \ref{lemma:absfil} with $\mathbb X=X$, $\mathbb Y=Y$, $T=T^{\boldsymbol\theta,w}$, $E=\mathrm{Id}_{X\times Y}$, $\mathbb Z=X\times Y$,
given $\mu,\nu\in\mathscr P_2(X)$, $\boldsymbol\eta\in\mathscr P(X\times Y)$ with $\mu=\mathrm{pr}_1\sharp\boldsymbol\eta$, $\boldsymbol\pi\in \Pi_o(\mu,\nu)$ 
there exists $\boldsymbol{\hat\eta}\in \boldsymbol{S_G}^{\boldsymbol\theta,w}(\nu)$ such that 
\begin{align*}
W_2(\boldsymbol\eta,\boldsymbol{\hat\eta})\le&\left[\iint_{(X\times Y)\times (X\times Y)}d^2_{X\times Y}((x_1,\gamma_1),(x_2,\mathrm{Pr}^{\boldsymbol\theta,w}(x_2,\gamma_1)))\,d\boldsymbol\xi (x_1,\gamma_1,x_2)\right]^{1/2}\\
\le&\left(\iint_{X\times X}|x_1-x_2|^2\,d\boldsymbol\pi(x_1,x_2)\right)^{1/2}+\left(\iint_{X\times Y}\mathrm{dist}^2(\gamma,S_G^{\boldsymbol\theta,w}(x))\,d\boldsymbol\eta(x,\gamma)\right)^{1/2}\\
=&W_2(\mu,\nu)+\left(\iint_{X\times Y}\mathrm{dist}^2(\gamma,S_G^{\boldsymbol\theta,w}(x))\,d\boldsymbol\eta(x,\gamma)\right)^{1/2},
\end{align*}
where $\boldsymbol\xi\in\mathscr P(X\times Y\times X)$ satisfies $\mathrm{pr}_{12}\sharp\boldsymbol\xi=\boldsymbol\eta$ and $\mathrm{pr}_{13}\sharp\boldsymbol\xi=\boldsymbol\pi$.

In particular, if $(y,\xi)=\mathrm{pr}_2\circ T(x_1,\gamma_1,x_2)$, we have $y=x_2$, $\xi\in S_G^{\boldsymbol\theta,w}(x_2)$ and
\[d_{Y}(\gamma_1,\xi)=\mathrm{dist}\left(\gamma_1, S_G^{\boldsymbol\theta,w}(x_2)\right).\]
Take the equivalent norm  on $\mathbb R^d\times C^0$
\[\|(z_1,\zeta_1)-(z_2,\zeta_2)\|_{\mathbb R^d\times C^0}\le |z_1-z_2|+\sup_{t\in I}\{e^{-Mt}|\zeta_1(t)-\zeta_2(t)|\},\]
where $M$ is a sufficiently large constant in order to have 
\[L_M:=\sup_{t\in I}\left\{e^{-Mt}\int_a^tk(s)\,ds\cdot\exp\left(\int_a^t k(s)\,ds\right)\right\}<1/2.\]

Suppose now that $x_1,x_2\in X$, $\gamma_1\in AC(I;X)$, and $(y,\xi)=\mathrm{pr}_2\circ T^{\boldsymbol\theta,w}(x_1,\gamma_1,x_2)$.
By Filippov's Theorem for Differential Inclusions we have
\begin{multline*}
e^{-Mt}|\gamma_1(t)-\xi(t)|\le\\ \le e^{-Mt}\exp\left(\int_a^t k(s)\,ds\right) \cdot \left[|\gamma_1(0)-\xi(0)|+\int_a^t \mathrm{dist}\left(\dot\gamma(s), G(s,w(s),\gamma_1(s),\theta_s)\right)\,ds\right],\end{multline*}
leading to
\begin{align*}
\left\|\left(\mathrm{pr}_{12}-\mathrm{pr}_2\circ T^{\boldsymbol\theta,w}\right)(x_1,\gamma_1,x_2)\right\|_{\mathbb R^d\times C^0}\le\hspace{-4cm}&\\
\le&|x_1-x_2|+\dfrac 12 |\gamma_1(0)-\xi(0)|+\\
&+\sup_{t\in I}\left\{e^{-Mt}\exp\left(\int_a^t k(s)\,ds\right) \cdot \int_a^t \mathrm{dist}\left(\dot\gamma(s), G(s,w(s),\gamma_1(s),\theta_s)\right)\,ds\right\}.
\end{align*}

In particular, if $\gamma_1\in S^{\boldsymbol{\hat \theta},\hat w}(x_1)$, we have $\gamma_1(0)=x_1$ and for a.e. $s\in I$  
\[\mathrm{dist}\left(\dot\gamma(s), G(s,w(s),\gamma_1(s),\theta_s)\right)\le \mathrm{dist}_H\left(G(s,\hat w(s),\gamma_1(s),\hat\theta_s), G(s,w(s),\gamma_1(s),\theta_s)\right),\]
and therefore recalling assumption $(D_G)$
\begin{align*}
&\left\|\left(\mathrm{pr}_{12}-\mathrm{pr}_2\circ T^{\boldsymbol\theta,w}\right)(x_1,\gamma_1,x_2)\right\|_{\mathbb R^d\times C^0}\le\\
&\le\dfrac 32 |x_1-x_2|+\sup_{t\in I}\left\{e^{-Mt}\exp\left(\int_a^t k(s)\,ds\right) \cdot \int_a^t k(s) \left(|\hat w(s)-w(s)|+W_2(\hat\theta_s,\theta_s)\right)\,ds\right\}\\
&\le\dfrac 12 \left(3|x_1-x_2|+\|\hat w-w\|_{\infty}+\sup_{s\in I}W_2(\hat\theta_s,\theta_s)\right).
\end{align*}
%

According to Lemma \ref{lemma:absfil}, by choosing $T=T^{\boldsymbol\theta,w}$, given $\mu,\nu\in\mathscr P(X)$, $\boldsymbol\eta\in\mathscr P(X\times Y)$, $\boldsymbol\pi\in \Pi_o(\mu,\nu)$ such that 
\begin{itemize}
\item[a.)] $\mu=\mathrm{pr}_1\sharp\boldsymbol\eta$;
\item[b.)] for $\boldsymbol\eta$-a.e. $(x,\gamma)\in X\times Y$ we have that $\gamma\in AC(I;X)$, 
\end{itemize}
there exists $\boldsymbol{\hat \eta}\in \mathscr P(X\times Y)$ satisfying
\begin{itemize}
\item $\boldsymbol{\hat\eta}\in \boldsymbol{S_G}^{\boldsymbol{\theta},w}(\nu)$, so $\left\{e_t\sharp \boldsymbol{\hat\eta}\right\}_{t\in I}\in \boldsymbol{\Upsilon_G}^{\boldsymbol{\theta},w}(\nu)$ by definition
\item it holds
\begin{align*}
W_2(\boldsymbol\eta,\boldsymbol{\hat\eta})\le&\dfrac{3}{2} \left[\iint_{X\times X}|x_1-x_2|^2\,d\boldsymbol\pi(x_1,x_2)\right]^{1/2}+\dfrac{1}{2}\left[\|\hat w-w\|_{\infty}+\sup_{s\in I}W_2(\hat\theta_s,\theta_s)\right]\\
=&\dfrac 32 W_2(\mu,\nu)+\dfrac{1}{2}\left[\|\hat w-w\|_{\infty}+\sup_{s\in I}W_2(\hat\theta_s,\theta_s)\right].
\end{align*}
\item it holds 
$W_2(e_t\sharp\boldsymbol\eta,e_t\sharp\boldsymbol{\hat\eta})\le W_2(\boldsymbol\eta,\boldsymbol{\hat\eta})$, since $e_t$ is $1$-Lipschitz continuous.
\end{itemize}
By taking in particular $\mu=\nu$, we have that the set-valued map $(w,\boldsymbol\theta)\mapsto\boldsymbol{\Upsilon_G}^{\boldsymbol{\theta},w}(\mu)$
is $1/2$-Lipschitz continuous with respect to the Hausdorff distance.
\par\medskip\par
Similarly, given any $\hat\gamma\in AC(I;\mathbb R^d)$ with $\hat\gamma(0)=x$, according to classical Filippov Theorem
there is a curve $\gamma\in S_F^{\boldsymbol\theta^{(1)},w_1}(x)$ such that 
\[|\hat\gamma(t)-\gamma(t)|\le \exp\left(\int_a^t k(s)\,ds\right) \cdot \int_a^t \mathrm{dist}\left(\dot{\hat\gamma}(s), F(s,w_1(s),\theta^{(1)}_s)\right)\,ds.\]
In particular, if $\gamma\in S_F^{\boldsymbol{\theta^{(2)}},w_2}(x)$ it holds
\[|\hat\gamma(t)-\gamma(t)|\le \exp\left(\int_a^t k(s)\,ds\right) \cdot \int_a^t \mathrm{dist}\left(F(s,w_2(s),\theta^{(2)}_s), F(s,w_1(s),\theta^{(1)}_s)\right)\,ds,\]
and so, as before, 
\begin{align*}
\mathrm{dist}&(\hat\gamma,S_1^{\boldsymbol\theta,w_1}(x))\le\sup_{t\in I}\left\{e^{-Mt}|\hat\gamma(t)-\gamma(t)|\right\}\\
\le&\sup_{t\in I} \left\{e^{-Mt} \exp\left(\int_a^t k(s)\,ds\right) \cdot \int_a^t \mathrm{dist}\left(F(s,w_2(s),\theta^{(2)}_s), F(s,w_1(s),\theta_s)\right)\,ds\right\},\\
\le&\dfrac 12 \left(\|w_2-w_1\|_{\infty}+\sup_{s\in I}W_2(\hat\theta^{(2)}_s,\theta^{(1)}_s)\right).
\end{align*}
In the same way, it can be proved that given any $\gamma\in S_F^{\boldsymbol{\theta^{(1)}},w_1}(x)$ it holds
\[\mathrm{dist}(\gamma,S_F^{\boldsymbol{\theta^{(2)}},w_2}(x))\le \dfrac 12 \left(\|w_2-w_1\|_{\infty}+\sup_{s\in I}W_2(\hat\theta^{(2)}_s,\theta^{(1)}_s)\right),\]
thus the set valued map $(w,\boldsymbol\theta)\mapsto S_F^{\boldsymbol{\theta},w}(x)$
is Lipschitz continuous with respect to the Hausdorff distance, with Lipschitz constant less or equal to $\dfrac 12$.
\par\medskip\par
The set valued map
\[(w,\boldsymbol\theta)\mapsto S_F^{\boldsymbol{\theta},w}(x)\times \boldsymbol{\Upsilon_G}^{\boldsymbol{\theta},w}(\mu).\]
is Lipschitz continuous with constant less than $\dfrac 12<1$ from the space 
\[\{w\in \Gamma_I:\, w(0)=x\}\times \{\boldsymbol\theta\in C^0(I;\mathscr P(\mathbb R^d)):\, \theta_0=\mu\},\]
into the space of its closed and bounded subsets.
\par\medskip\par
By classical Nadler's contraction theorem (see Theorem 5 in \cite{Nad}), the set valued map above admits a fixed point, i.e.,
there exists $(\gamma,\boldsymbol\mu)\in \Gamma_I\times C^0(I;\mathscr P(\mathbb R^d))$ such that 
$\gamma\in S_F^{\boldsymbol\mu,\gamma}(x)$ and $\boldsymbol\mu\in\Upsilon_G^{\boldsymbol\mu,\gamma}(\mu)$.
Nadler's theorem provides also that the set valued map of the fixed points has closed graph (see Theorem 9 in \cite{Nad}).
\par\medskip\par
We prove now that the set of solutions $\mathscr A_I(x,\bar\mu)$ of \eqref{def:diffprob} coincides with the set of fixed points
\[(\gamma,\boldsymbol\mu)\in {S_F}^{\boldsymbol{\mu},\gamma}(\bar x)\times \boldsymbol{\Upsilon_G}^{\boldsymbol{\mu},\gamma}(\mu).\]
Indeed, given $(\gamma,\boldsymbol\mu)\in\mathscr A(x,\bar\mu)$, we have that $\gamma\in S_F^{\boldsymbol\mu,\gamma}(\bar x)$,
and $\boldsymbol\mu=\{\mu_t\}_{t\in I}$ is a solution of the continuity equation $\partial_t\mu_t+\mathrm{div}(v_t\mu_t)=0$
satisfying $\mu_0=\bar \mu$ and $v_t(x)\in G(t,\gamma(t),x,\mu_t)$. 
By Theorem 8.3.1 in \cite{AGS}, there exists $\boldsymbol\eta\in\mathscr P(\mathbb R^d\times\Gamma_I)$ supported on pairs
$(y,\xi(\cdot))$ such that $\xi(0)=y$ and $\dot \xi(t)=v_t(\gamma(t))\in G(t,\gamma(t),\xi(t),\mu_t)$. 
This amounts to say that $\boldsymbol\eta$ is supported on $\mathrm{graph}({S_G}^{\boldsymbol{\mu},w})$,
thus $\boldsymbol\eta\in\boldsymbol{S_G}^{\boldsymbol{\mu},\gamma}(\bar\mu)$ and so $\boldsymbol\mu\in\boldsymbol{S_G}^{\boldsymbol{\mu},\gamma}(\bar\mu)$.
Since $\gamma\in{S_G}^{\boldsymbol{\mu},\gamma}(\bar x)$, we have that $(\gamma,\boldsymbol\mu)$ is a fixed point.
\par\medskip\par
Conversely, suppose that $(\gamma,\boldsymbol\mu)$ is a fixed point. Clearly we have that $\dot\gamma(t)\in F(t,\gamma(t),\mu_t)$
for a.e. $t\in I$. Furthermore, $\boldsymbol\mu=\{\mu_t\}_{t\in I}$ satisfies $\mu_t=e_t\sharp\boldsymbol\eta$
for a certain $\boldsymbol\eta$ supported on $(x,\xi)$ where $\xi\in S_G^{\boldsymbol\mu,\gamma}(x)$.
Defined 
\[v_t(y)=\int_{\{\gamma:\,\gamma(t)=y\}} \dot\gamma(t)\,d\eta_{t,y}(x,\gamma),\]
where $\boldsymbol\eta=\mu_t\otimes\eta_{t,y}$ is the disintegration of $\boldsymbol\eta$ w.r.t. $e_t$,
we have that $(t,y)\mapsto v_t(y)$ is a Borel selection of $(t,y)\mapsto G(t,\gamma(t),y,\mu_t)$ 
for a.e. $t\in I$ and $\mu_t$-a.e. $y\in\mathbb R^d$, due to the convexity of the images of $G$.
For every $\varphi\in C^1_c(]0,T[\times\mathbb R^d)$ it holds
\begin{align*}
\int_I\int_{\mathbb R^d}&[\partial_t\varphi(x,t)+\langle v_t(x),\nabla_x\varphi(x,t)\rangle]\,d\mu_t(x)\,dt=\\
=&\int_I\int_{\mathbb R^d}[\partial_t\varphi(x,t)+\langle \int_{\{\gamma:\,\gamma(t)=x\}} \dot\gamma(t)\,d\eta_{t,x}(z,\gamma),\nabla_x\varphi(x,t)\rangle]\,d(e_t\sharp\boldsymbol\eta)(x)\,dt\\
=&\int_I\int_{\mathbb R^d}\int_{\{\gamma:\,\gamma(t)=x\}}[\partial_t\varphi(\gamma(t),t)+\langle\dot\gamma(t),\nabla_x\varphi(\gamma(t),t)\rangle]\,d\eta_{t,x}(z,\gamma)\,d(e_t\sharp\boldsymbol\eta)(x)\,dt\\
=&\int_I\iint_{\mathbb R^d\times\Gamma_I}[\partial_t\varphi(\gamma(t),t)+\langle\dot\gamma(t),\nabla_x\varphi(\gamma(t),t)\rangle]\,d\boldsymbol\eta(z,\gamma)\,dt\\
=&\int_I\iint_{\mathbb R^d\times\Gamma_I}\dfrac{d}{dt}[\varphi(\gamma(t),t)]\,d\boldsymbol\eta(z,\gamma)\,dt\\
=&\iint_{\mathbb R^d\times\Gamma_I}[\varphi(\gamma(T),T)-\varphi(\gamma(0),0)]\,d\boldsymbol\eta(z,\gamma)=0,
\end{align*}
since $\varphi$ has compact support contained in $]0,T[\times\mathbb R^d$.
Thus $\partial_t\mu_t+\mathrm{div}(v_t\mu_t)=0$ in the sense of distributions, 
and therefore $(\gamma,\boldsymbol\mu)\in\mathscr A_I(\bar x,\bar \mu)$.
\end{proof}
\begin{proposition}\label{prop:filippov}
	Let $I=[0,T]$. Under the assumptions $(D_F)$ and $(D_G)$, given $(\bar x,\bar\mu)\in X\times\mathscr P_2(X)$, a Borel map $\bar v(\cdot)\in L^2_{\bar\mu}(X)$ 
	with $\bar v(x)\in G(0,x,\bar x,\bar\mu)$ for all $x\in X$, and $v_0\in \mathbb R^d$, there exists $(\gamma,\boldsymbol\mu)\in \mathscr A_I(\bar x,\bar \mu)$
	and $\boldsymbol\eta\in\mathscr P(\Gamma_I)$ and a right neighborhood $J$ of $0$ in $I$ with the following properties
	\begin{itemize}
		\item $\gamma(0)=\bar x$, $\gamma_{|J}\in C^1(\bar J)$, $\dot\gamma(0)=v_0$;
		\item $\boldsymbol\mu=\{\mu_t\}_{t\in I}$, $\mu_t=e_t\sharp\boldsymbol\eta$, $\boldsymbol\eta\in \boldsymbol{S_G}^{\boldsymbol\mu,\gamma}(\bar\mu)$;
		\item for $\boldsymbol\eta$-a.e. $\xi\in \Gamma_I$ we have that $\xi_{|J}\in C^1(\overline{J})$ and $\dot\xi(0)=\bar v(\xi(0))$.
	\end{itemize}
\end{proposition}
\begin{proof}
	The proof is postponed to Appendix for convenience of the reader.
\end{proof}

\section{HJB equation}\label{sec:HJB}

Assume that $F$ and $G$ satisfy the assumptions in the Section \ref{sec:model}. 
We consider the following differential problem 
\begin{align}\begin{cases}\label{eq:system1}
		\dot \xi(t)\in F(t,\xi(t),\mu_t),&\textrm{ for a.e. $t\in I$},\\
		\partial_t \mu_t+\mathrm{div}(v_t\mu_t)=0,\\
		\xi(t_0)=\xi_0, \ \ \ \mu_{t_0}=\mu_0
\end{cases}\end{align}
where the coupling between the two equations is given by assuming that the vector field $v=v_t(y)$ 
can be any Borel selection of the set-valued map $(t,y)\mapsto G(t,\xi(t),y,\mu_t)$. We denote by $S(t_0,\xi_0,\mu_0)$ the set of solution to \eqref{eq:system1}. 

\begin{definition}
Let $\mathscr  L: \mathbb T^d \times  \mathscr P_2(\mathbb T^d) \to [0,+\infty[$ and $\mathscr  G : \mathbb T^d \times  \mathscr P_2 (\mathbb T^d) \to [0, +\infty[$ be Borel measurable. Define $V:[0, T ] \times  \mathbb T^d \times  \mathscr P_2 (\mathbb T^d) \to\mathbb R$ to be the value function of the minimization problem
\begin{equation}\label{eq:vf}
V(t_0,\xi_0,\mu_0):= \inf_{(\xi(\cdot),\mu_\cdot)\in \mathscr A_{[t_0,T]}(\xi_0,\mu_0)} J(t_0,\xi_0,\mu_0,\xi(\cdot),\mu_\cdot),
\end{equation}
where 
\[J(t_0,\xi_0,\mu_0,\xi(\cdot),\mu_\cdot)= \int_{t_0}^T \mathscr  L(\xi(s),\mu_s)\,ds +\mathscr  G(\xi(T),\mu_T).\]
\end{definition}
On $\mathscr L$, $\mathscr G$ we make the following assumption
\begin{equation}\tag{$A$}
\textrm{$\mathscr L$, $\mathscr G$ are bounded and Lipschitz continuous.}
\end{equation}

It is well known that under assumption $(A)$ the map $V$ is bounded and Lipschitz continuous, 
furthermore the infimum appearing in \eqref{eq:vf} is attained due to the compactness of $\mathscr A_{[t_0,T]}(\xi_0,\mu_0)$.
The Dynamical Programming Principle holds, as shown in the following Proposition.
\begin{proposition}[Dynamical Programming Principle]\label{prop:DPP}
The value function $V$ defined in \eqref{eq:vf} satisfies the Dynamical Programming Principle, i.e., 
\begin{equation}
V(t_0,\xi_0,\mu_0)= \inf \left\{ \int_{t_0}^\tau \mathscr L(\xi(s),\mu_s)\,ds + V(\tau,\xi(\tau),\mu_\tau), \ (\xi,\mu) \in \mathscr A_{[t_0,T]}(\xi_0,\mu_0)\right\},
\end{equation}
for any $\tau \in[t_0,T]$.
\end{proposition}

\begin{proof}
The proof is standard and it is postponed to the Appendix for the convenience of the reader.
\end{proof}

\begin{proposition}
The value function $V$ defined by \eqref{eq:vf} is a subsolution of \eqref{eq:HJB} with Hamiltonian function $H$ defined by
\begin{multline}\label{eq:definitionham}
H(t,\xi,\mu,p_x,p_\mu(\cdot)):=\\\inf\left\{\mathscr  L(\xi,\mu)+\langle v,p_x\rangle +\langle p_\mu,w\rangle_{L^2_\mu}  : v\in F(t,\xi,\mu) \textrm{ and $\mu$-a.e. it holds}\ w(\cdot)\in G(t,\xi,\cdot,\mu)\right\}.
\end{multline}
\end{proposition}
\begin{proof}
Take $0\leq t_0 <t_0+h<T$ and fix $\xi_0,\mu_0 \in \mathbb{T}^d \times \mathscr{P}_2(\mathbb{T}^d)$, we consider $(p_{t_0},p_{x_0},p_{\mu_0})\in\partial^+ V(t_0,\xi_0,\mu_0)$. Let 
\begin{equation*}
(v,w(\cdot))\in \mbox{argmin} \left\{\langle v,p_x\rangle +\langle p_\mu,w\rangle_{L^2_\mu}  : v\in F(t,\xi,\mu) \ \mbox{and}\ w\in G(t,\xi,\cdot,\mu) \mu-\mbox{a.e.}\right\}.
\end{equation*}
By the corollary of Filippov's Theorem \ref{prop:filippov}, there exists $(\xi(\cdot),\mu_\cdot)\in \mathscr A_{[t_0,T]}(\xi_0,\mu_0)$ and $\boldsymbol\eta\in \mathscr P(\mathbb T^d\times \Gamma_I)$
such that $\mu_t=e_t\sharp\boldsymbol\eta$ and 
\begin{align}
&\dot{\xi}(0)=v\\
& \int_{\mathbb{T}^d} \langle p_{\mu_0}(x), w(t)\rangle d\,\mu_0(x) = \lim_{h\rightarrow 0^+}\int_{\mathbb{T}^d\times \Gamma_{[t_0,T]}} \langle p_{\mu_0}(x), \frac{\gamma(t+h)-\gamma(t)}{h}\rangle \ , d\eta(x,\gamma)\label{eq:initvel}
\end{align}
Take $\pi_h:=(e_0,e_h)\sharp\boldsymbol\eta$. We notice that $\{\pi_h\}_{h\in [t_0,T]}$ is a family of admissible variations from $\mu_0$ by assumption $(A)$. 
By Proposition \ref{prop:DPP}, we have
\begin{align*}
0&\leq V(t_0+h,\xi(t_0+h),\mu_{t_0+h})-V(t_0,\xi_0,\mu_0)+ \int_{t_0}^{t_0+h}\mathscr L(\xi(s),\mu_s)\,ds \\
&\leq p_{t_0}h+ h \langle p_{x_0},v\rangle +\iint_{\mathbb T^d\times\mathbb T^d} \langle p_{\mu_0}(x),\zeta-x\rangle \,d\pi_h(x,\zeta)+ \int_{t_0}^{t_0+h} \mathscr  L(\xi(s),\mu_s)\,ds\\
&\leq h\cdot \left[p_{t_0}+\langle p_{x_0},v\rangle+ \int_{\mathbb T^d} \langle p_{\mu_0}(x), \dfrac{\gamma(t+h)-\gamma(t)}{h}\rangle \, d\boldsymbol\eta(x,\gamma)+ \dfrac{1}{h}\int_{t_0}^{t_0+h}\mathscr  L(\xi(s),\mu_s)\,ds\right]
\end{align*}
Dividing by $h$ the above inequality, recalling  \eqref{eq:initvel}, the continuity of $\mathcal L$, and the definition of Hamiltonian function, by passing to the liminf we get
\[0\leq p_{t_0}+\langle p_{x_0},v\rangle+ \int_{\mathbb T^d}\langle p_{\mu_0}(x),w(x)\rangle\,d\mu_0+\mathscr L(\xi_0,\mu_0)= p_{t_0}+ H(t,x_0,\xi_0,p_{t_0},p_{x_0},p_{\mu_0}),\]
as desired.
\end{proof}

		\begin{proposition}
		The value function $V$ defined by \eqref{eq:vf} is a supersolution of HJB \eqref{eq:definitionham}
	in the sense of Definition \ref{def:viscosol}.
	\end{proposition}

\begin{proof} Take $0\leq t_0 <t_0+h<T$ and fix $\xi_0,\mu_0 \in \mathbb{T}^d \times \mathscr{P}_2(\mathbb{T}^d)$, we consider $(p_{t_0},p_{x_0},p_{\mu_0})\in\partial^- V(t_0,\xi_0,\mu_0)$
For each $\varepsilon>0$, there exists an optimal pair 
$(\xi^\varepsilon,\mu^\varepsilon)$ such that
\begin{equation}\label{eq:DPP-eps}
	V(t_0+h,\xi^\varepsilon(t_0+h),\mu^\varepsilon_{t_0+h}) - V(t_0,\xi_0,\mu_0)
	\;\le\; - \int_{t_0}^{t_0+h} L(\xi^\varepsilon(s),\mu^\varepsilon_s)\,ds + \varepsilon h.
\end{equation}
Since $(p_{t_0},p_{\xi_0},p_{\mu_0}) \in \partial^- V(t_0,\xi_0,\mu_0)$, for the coupling 
$\pi_h^\varepsilon$ associated with $(\xi^\varepsilon,\mu^\varepsilon)$ we have
\begin{align}\label{eq:subjet}
	&V(t_0+h,\xi^\varepsilon(t_0+h),\mu^\varepsilon_{t_0+h}) - V(t_0,\xi_0,\mu_0)
	\\
	&\geq p_{t_0}h + \langle p_{\xi_0}, \xi^\varepsilon(t_0+h)-\xi_0\rangle
	+ \int_{\mathbb T^d\times\mathbb T^d}\!\! \langle p_{\mu_0}(x),\zeta-x\rangle\,d\pi_h^\varepsilon(x,\zeta) - o(h).
\end{align}
Comparing \eqref{eq:DPP-eps} and \eqref{eq:subjet} yields
\begin{align*}
	& p_{t_0}h + \langle p_{\xi_0}, \xi^\varepsilon(t_0+h)-\xi_0\rangle
	+ \int_{\mathbb T^d\times\mathbb T^d}\!\! \langle p_{\mu_0}(x),\zeta-x\rangle\,d\pi_h^\varepsilon(x,\zeta) \\
	& \qquad + \int_{t_0}^{t_0+h} L(\xi^\varepsilon(s),\mu^\varepsilon_s)\,ds
	\;\le\; \varepsilon h + o(h).
\end{align*}
Divide by $h$ and send $h\downarrow 0$. Since $(\xi^\varepsilon,\mu^\varepsilon)$ is an optimal pair
, we can select $v\in F(t_0,\xi_0,\mu_0)$ and 
$w(\cdot)\in G(t_0,\xi_0,\cdot,\mu_0)$ such that
\[
\frac{\xi^\varepsilon(t_0+h)-\xi_0}{h}\to v, 
\qquad 
\frac{1}{h}\!\int\!\langle p_{\mu_0}(x),\zeta-x\rangle\,d\pi_h^\varepsilon
\to \int_{\mathbb T^d}\!\langle p_{\mu_0}(x),w(x)\rangle\,d\mu_0(x),
\]
and
\[
\frac{1}{h}\int_{t_0}^{t_0+h} L(\xi^\varepsilon(s),\mu^\varepsilon_s)\,ds \to L(\xi_0,\mu_0).
\]
Hence
\[
p_{t_0} + \langle p_{\xi_0},v\rangle
+ \int_{\mathbb T^d}\!\langle p_{\mu_0}(x),w(x)\rangle\,d\mu_0(x) 
+ L(\xi_0,\mu_0) \;\le\; \varepsilon.
\]
Since $\varepsilon>0$ is arbitrary,by passing to the liminf  we conclude
\[
p_{t_0} + H\bigl(t_0,\xi_0,\mu_0;p_{\xi_0},p_{\mu_0}\bigr) \;\le\; 0.
\]

\end{proof}

\begin{definition}
For every measure $\mu$ and $\nu$ and $\pi\in \Pi(\mu,\nu)$ we define $p_{\mu, \nu}^\pi \in {L}^2_\mu(\mathbb{T}^d)$ such that
\begin{equation}
\int_{\mathbb{T}^d} \langle \phi(x),x-y\rangle \,d\pi(x,y)=\int_{\mathbb{T}^d} \langle \phi(x), p^\pi (x)\rangle \,d\mu(x)
\end{equation}
for all $\phi \in {L}^2_\mu(\mathbb{T}^d)$.
\end{definition}
By Lemma 4 \cite{JMQ}, we recall that 
\begin{equation*}
p^\pi (x)=x-\int_{\mathbb{T}^d} y\,d\pi_x(y),
\end{equation*}
where $\pi$ is disintegrated as $\pi=(pr_1\sharp \pi)\otimes \pi_x$.
\begin{remark}\label{rem:invplan}
If $\pi\in \Pi_0(\mu,\nu)$ then $\pi^{-1}\in \Pi_0(\nu,\mu)$. Indeed, for all $\phi\in C^0(\mathbb{T}^d)$ it holds: 
\begin{align*}
\int_{\mathbb{T}^d} \langle \phi(y), p^{\pi^{-1}}(y)\rangle \,d\nu(y)&=  \int_{\mathbb{T}^d} \langle \phi(y),y-\int_{\mathbb{T}^d}x\,d\pi^{-1}_y(x)\rangle \,d\nu(y)=\iint_{\mathbb{T}^d\times \mathbb{T}^d} \langle \phi(y),y-x\rangle \,d\pi^{-1}(x,y)\\
&=\iint_{\mathbb{T}^d\times \mathbb{T}^d}\langle \phi(x),x-y\rangle \,d\pi(x,y)=\int_{\mathbb{T}^d}\langle \phi(x),p^{\pi}(x)\rangle \,d \mu(x)
\end{align*}
\end{remark}


\begin{proposition}\label{prop:estimateH}
Let $X:=[0,T]\times \mathbb{T}^d\times \mathscr{P}(\mathbb{T}^d)$.
For every $\mu$, $\nu$, for every $\pi \in \Pi_0(\mu,\nu)$ and for every $\lambda >0$ it holds the following
\begin{align*}
H(t,x,\mu,p,\lambda p^\pi)-H(s,y,\nu,q,\lambda p^{\pi^{-1}})\leq C \Delta +k\Delta\psi|p|+ |\bar\eta||p-q| +\lambda \Delta W_2(\mu,\nu),
\end{align*}
where $\Delta= d_X((t,x,\mu),(s,y,\nu))$.
\end{proposition}

\begin{proof}
Let $\pi\in \Pi_0(\mu,\nu)$. Using assumption (A) and recalling the definition (\ref{eq:definitionham}) we have
\begin{align*}
&H(t,x,\mu,p,\lambda p^\pi)-H(s,y,\nu,q,\lambda p^{\pi^{-1}})\\
&\leq C(|x-y|+W_2(\mu,\nu)) + \inf_{\eta \in F(t,x,\mu), w(\cdot)\in G(t,x,\cdot , \mu)} \{\langle p,\eta\rangle+ \lambda\langle w,p^\pi\rangle\}  \\
&-  
\inf_{\bar\eta \in F(s,y,\nu), \bar w(\cdot)\in G(s,y,\cdot , \nu)} \{\langle q,\bar\eta\rangle+ \lambda\langle \bar w,p^{\pi^{-1}}\rangle\}
\end{align*}
Consider $(\bar{\eta},\bar{w}(\cdot))$ such that 
\[
\inf_{\bar\eta \in F(s,y,\nu), \bar w(\cdot)\in G(s,y,\cdot , \nu)} \{\langle q,\bar\eta\rangle+ \lambda\langle \bar w,p^{\pi^{-1}}\rangle\}=\langle q,\bar\eta\rangle+ \lambda\langle \bar w,p^{\pi^{-1}}\rangle.
\]

Recalling the Lipschitz continuity of $F$, $G$, it is possible to find $\psi\in\mathbb R^d$, $|\psi|\le 1$ and a measurable
$\xi(\cdot)$ with $|\xi(x)|\le 1$ for all $x\in X$ such that $(\eta,w(\cdot))=(\bar{\eta}+k\Delta\psi, \bar w(\cdot)+k\Delta, \xi )$, where $\Delta=d_X((t,x,\mu),(s,y,\nu))$. Using Remark \ref{rem:invplan}, we get
{\small\begin{align*}
	&H(t,x,\mu,p,\lambda p^\pi)-H(s,y,\nu,q,\lambda p^{\pi^{-1}})\\
	&\leq  C(|x-y|+W_2(\mu,\nu)) +\langle k\Delta\psi,p\rangle + \langle p-q,\bar\eta\rangle +\lambda \int_{\mathbb{T}^d} \langle  w(x),p^\pi \rangle \, d \mu(x)- \lambda \int_{\mathbb{T}^d} \langle  \bar w(y),p^{\pi^{-1}} \rangle \, d \nu(y)\\
	&\leq C(|x-y|+W_2(\mu,\nu)) +k\Delta\psi|p|+ |\bar\eta||p-q| +\lambda \int_{\mathbb{T}^d} \langle  w(x),p^\pi \rangle \, d \mu(x)- \lambda \int_{\mathbb{T}^d} \langle  \bar w(x),p^{\pi} \rangle \, d \mu(x)\\
&=C(|x-y|+W_2(\mu,\nu)) +k\Delta\psi|p|+ |\bar\eta||p-q| + \lambda \Delta \int_{\mathbb{T}^d \times \mathbb{T}^d} \langle  \xi(x),x-y \rangle \, d \pi(x,y)\\
&\leq C \Delta +k\Delta\psi|p|+ |\bar\eta||p-q| +\lambda \Delta W_2(\mu,\nu).
\end{align*}}	
\end{proof}
\begin{theorem}[Comparison Principle]
Let $X:=[0,T]\times \mathbb{T}^d\times \mathscr{P}(\mathbb{T}^d)$.
Let $V_1,V_2:X\rightarrow \mathbb{R}$ be subsolution and supersolution  of  \eqref{eq:HJB}, respectively. If $V_1$ and $V_2$ are bounded Lipschitz continuous functions and $V_1(T,\cdot,\cdot)=V_2(T,\cdot,\cdot)$ then $V_1(\cdot,\cdot,\cdot)\leq V_2(\cdot,\cdot,\cdot)$
\end{theorem}
\begin{proof}
By contradiction we suppose that 
\[-\xi := \min_{X} V_2-V_1 <0,\]
which is achieved at $(t_0,\xi_0,\mu_0)$. We observe that $t_0\neq T$ always holds. \par
For $\epsilon\in]0, \min\{1,\frac{1}{16\mathrm{Lip}(V_1)}\}[$, there exists $\eta \in ]0,\eta^*(\epsilon)[$ such that 
\begin{itemize}
\item[(a)] $\eta T\leq \frac{\xi}{4}$
\item[(b)] $\epsilon(\eta+Lip(V_1))Lip(V_1)\leq \frac{\xi}{2}$
\item[(c)] $4(\eta+Lip(V_1))^2\epsilon<\eta$.
\end{itemize}
We define $\phi_{\epsilon, \eta}:X\rightarrow \mathbb{R}$ as follows
\begin{align}
\phi_{\epsilon, \eta}((t,x,\mu),(s,y,\nu)):= V_2(t,x,\mu)-V_1(s,y,\nu) + \frac{1}{2\epsilon} d_X^2((t,x,\mu),(s,y,\nu)) -\eta s
\end{align}
Since $\phi_{\epsilon,\eta}$ is continuous on the compact set X then $\phi_{\epsilon,\eta}$ has a minimum in\par
$z_{\epsilon,\eta}:=((t_{\epsilon,\eta},x_{\epsilon,\eta},\mu_{\epsilon,\eta}),(s_{\epsilon,\eta},y_{\epsilon,\eta},\nu_{\epsilon,\eta}))$.
Set $\rho_\epsilon:=d_X ((t_{\epsilon,\eta},x_{\epsilon,\eta},\mu_{\epsilon,\eta}),(s_{\epsilon,\eta},y_{\epsilon,\eta},\nu_{\epsilon,\eta}))$. Our goal is to estimate $\phi_{\epsilon,\eta}(z_{\epsilon,\eta})$. So,
\begin{align*}
\phi_{\epsilon,\eta}(z_{\epsilon,\eta})&= V_2(t_{\epsilon,\eta},x_{\epsilon,\eta},\mu_{\epsilon,\eta})- V_1(s_{\epsilon,\eta},y_{\epsilon,\eta},\nu_{\epsilon,\eta}) + \frac{1}{2\epsilon}\rho_\epsilon^2 -\eta s_{\epsilon,\eta}\\
& \leq \phi_{\epsilon,\eta} ((t_{\epsilon,\eta},x_{\epsilon,\eta},\mu_{\epsilon,\eta}),(t_{\epsilon,\eta},x_{\epsilon,\eta},\mu_{\epsilon,\eta}))= V_2(t_{\epsilon,\eta},x_{\epsilon,\eta},\mu_{\epsilon,\eta})-V_1(t_{\epsilon,\eta},x_{\epsilon,\eta},\mu_{\epsilon,\eta})- \eta t_{\epsilon,\eta}
\end{align*}
Recalling that $V_1$ is bounded Lipschitz continuous function , we get
\begin{align*}
	\rho_\epsilon^2 \leq 2\epsilon [ \eta (s_{\epsilon,\eta}-t_{\epsilon,\eta})+  \mbox{Lip}(V_1)(\rho_\epsilon)]\leq  2\epsilon(\eta +\mbox{Lip}(V_1))\rho_\epsilon 
\end{align*}
and so
\begin{align}\label{eq:estrho}
\rho_\epsilon \leq 2(\eta+Lip(V_1))\epsilon.
\end{align}

Moreover, recalling that $\phi_{\epsilon,\eta}$ has a minimum in $z_{\epsilon,\eta}$, we have that
\begin{align*}
\phi_{\epsilon,\eta}(z_{\epsilon,\eta})\leq V_2(t_0,x_0,\mu_0)-V_1(t_0,x_0,\mu_0)-\eta t_0=-\xi -\eta t_0.
\end{align*}
Using the $V_1$ is Lupschitz continuous and \eqref{eq:estrho}, we get
\begin{align*}
&0\leq \frac{1}{2\epsilon} \rho^2_\epsilon \leq -\xi + \eta(s_{\epsilon,\eta}-t_0)+ V_1(t_{\epsilon,\eta},x_{\epsilon,\eta},\mu_{\epsilon,\eta})-V_2(t_{\epsilon,\eta},x_{\epsilon,\eta},\mu_{\epsilon,\eta})+ Lip(V_1)\rho_{\epsilon}\\
&\leq -\xi + \eta(s_{\epsilon,\eta}-t_0)+  \rho_\epsilon Lip(V_1)+ V_1(t_{\epsilon,\eta},x_{\epsilon,\eta},\mu_{\epsilon,\eta})-V_2(t_{\epsilon,\eta},x_{\epsilon,\eta},\mu_{\epsilon,\eta})\\
&\leq -\xi + \eta(s_{\epsilon,\eta}-t_0)+ 2\epsilon (\eta+Lip(V_1))Lip(V_1)+V_1(t_{\epsilon,\eta},x_{\epsilon,\eta},\mu_{\epsilon,\eta})-V_2(t_{\epsilon,\eta},x_{\epsilon,\eta},\mu_{\epsilon,\eta})
\end{align*}
Recalling the properties (a) and (b) of $\eta$ we have that
\begin{align*}	
\frac{\xi}{2}\leq V_1(t_{\epsilon,\eta},x_{\epsilon,\eta},\mu_{\epsilon,\eta})-V_2(t_{\epsilon,\eta},x_{\epsilon,\eta},\mu_{\epsilon,\eta})
\end{align*}
Consequently $t_{\epsilon, \eta}\neq T$, because the right hand side would be equal to zero. Similarly $s_{\epsilon, \eta}\neq T$.\par\medskip\par

Now, we know that 
\begin{equation}
\phi_{\epsilon,\eta}(z_{\epsilon,\eta})\leq \phi_{\epsilon,\eta}((t_{\epsilon,\eta},x_{\epsilon,\eta},\mu_{\epsilon,\eta}),(s,y,\nu)) \ \ \forall (s,y,\nu)
\end{equation}
and so
\begin{align} \label{eq:estcomp}
&V_1(s,y,\nu)-V_1(s_{\epsilon,\eta},y_{\epsilon,\eta},\nu_{\epsilon,\eta}) \leq -\frac{1}{2\epsilon}\rho_\epsilon^2 +\frac{1}{2\epsilon}d^2_x((t_{\epsilon,\eta},x_{\epsilon,\eta},\mu_{\epsilon,\eta}),(s,y,\nu))+\eta(s_{\epsilon,\eta}-s)\\
&\leq \frac{1}{2\epsilon} (|t_{\epsilon,\eta}-s|^2-|t_{\epsilon,\eta}-s_{\epsilon,\eta}|^2)+\frac{1}{2\epsilon}(|x_{\epsilon,\eta}-y|^2-|x_{\epsilon,\eta}-y_{\epsilon,\eta}|^2) \\
&+ \frac{1}{2\epsilon} (W_2^2(\nu,\mu_{\epsilon,\eta})-W^2_2(\nu_{\epsilon,\eta},\mu_{\epsilon,\eta}))+\eta(s_{\epsilon,\eta}-s)\nonumber\\
&=\frac{1}{\epsilon} (s_{\epsilon,\eta}-s)(t_{\epsilon,\eta}-s_{\epsilon,\eta} )+\frac{1}{2\epsilon} (s_{\epsilon,\eta}-s)^2+\frac{1}{\epsilon}\langle y_{\epsilon,\eta}-y,x_{\epsilon,\eta}-y_{\epsilon,\eta}\rangle +\frac{1}{2\epsilon}|y_\epsilon-y|^2+ \eta(s_{\epsilon,\eta}-s)\nonumber \\
&+ \frac{1}{2\epsilon} (W_2^2(\mu_{\epsilon,\eta},\nu)-W^2_2(\mu_{\epsilon,\eta},\nu_{\epsilon,\eta}))\nonumber
\end{align}
Take $\bar{\pi}\in \Pi_0(\mu_{\epsilon,\eta},\nu_{\epsilon,\eta})$ such that its disintegration is  $\bar{\pi}=\nu_{\epsilon,\eta}\otimes \pi^x$, where $\{\pi^x\}_{x\in \mathbb{T}^d}$ are Borel measures. Let $(\pi_h)_{h>0}\in \mathcal{P}(\mathbb{T}^d\times \mathbb{T}^d)$  such that $\mathrm{pr}_1 \sharp \pi_h=:\mu_h$ and $\mathrm{pr}_2 \sharp \pi_h=:\nu_{\epsilon,\eta}$.
Let define $\sigma \in \mathcal{P}(\mathbb{T}^d\times \mathbb{T}^d\times \mathbb{T}^d)$ such that $\sigma:= \mu_{\epsilon, \eta}\otimes \pi^x\otimes \pi_h^x$ with $\mathrm{pr}_{1,2}\sharp \sigma =\bar \pi$,  $\mathrm{pr}_{1,3}\sharp \sigma =\bar \pi_h$ and  $\mathrm{pr}_{2,3}\sharp \sigma =:\hat\pi_h \in \Pi(\mu_{\epsilon,\eta},\nu_h)$. Notice that  $\mathrm{pr}_1\sharp(\mathrm{pr}_{2,3}\sharp \sigma) =\mathrm{pr}_1\sharp \bar \pi=\nu_{\epsilon,\eta}$ and  $\mathrm{pr}_2\sharp(\mathrm{pr}_{2,3}\sharp \sigma)=\mu_h$. So $\hat \pi_h \in \Pi(\nu_{\epsilon,\eta},\mu_h)$.

Hence,
\begin{align*}
&W_2^2(\mu_h,\nu_{\epsilon,\eta})-W^2_2(\mu_{\epsilon,\eta},\nu_{\epsilon,\eta})\leq \int_{\mathbb{T}^d\times \mathbb{T}^d} |y-z|^2\,d\hat{\pi}_h(y,z)-\int_{\mathbb{T}^d\times \mathbb{T}^d} |x-y|^2\, d\bar{\pi}(x,y)\\
&= \int_{\mathbb{T}^d\times \mathbb{T}^d\times \mathbb{T}^d}|y-x+x-z|^2\,d\hat{\pi}_h(y,z)\,d\mu_{\epsilon,\eta}(x)
-\int_{\mathbb{T}^d\times \mathbb{T}^d} |x-y|^2\,d\bar{\pi}(x,y)\\
&\leq\int_{\mathbb{T}^d\times \mathbb{T}^d}|y-x|^2\,d\nu_{\epsilon,\eta}(y)\,d\mu_{\epsilon,\eta}(x)+\int_{\mathbb{T}^d\times \mathbb{T}^d} |x-z|^2\,d\mu_{h}(z)\,d\mu_{\epsilon,\eta}(x)\\
&+ 2 \int_{\mathbb{T}^d\times {\mathbb{T}^d\times \mathbb{T}^d}} \langle y-x,x-z\rangle\,d\pi_h(y,z)d\mu_{\epsilon,\eta}(x)-\int_{\mathbb{T}^d\times{\mathbb{T}^d}}|x-y|^2\, d \bar \pi(x,y)\\
&=\int_{\mathbb{T}^d\times \mathbb{T}^d} |x-z|^2\,d\mu_{h}(z)\,d\mu_{\epsilon,\eta}(x)
+2  \int_{\mathbb{T}^d\times \mathbb{T}^d} \langle y-x,x-\int_{\mathbb{T}^d}z\, d \pi_h^x(z) \rangle\,d\pi_h(x,y)
\end{align*}
The above estimate together with \eqref{eq:estcomp} gives
\begin{equation}
\left( \frac{1}{\epsilon}(s_{\epsilon,\eta}-t_{\epsilon,\eta})-\eta,\frac{y_{\epsilon,\eta}-x_{\epsilon,\eta}}{\epsilon}, \frac{1}{\epsilon}p^{\bar\pi}\right) \in \partial^+V_1(s_{\epsilon,\eta},y_{\epsilon,\eta},\nu_{\epsilon,\eta}),
\end{equation}
where $p^{\bar\pi}=\int_{\mathbb{T}^d}z\, d \pi_h^x(z)$.\par\medskip\par

In a similar way, we know that 
\begin{equation}
	\phi_{\epsilon,\eta}(z_{\epsilon,\eta})\leq \phi_{\epsilon,\eta}((t,x,\mu),((s_{\epsilon,\eta},y_{\epsilon,\eta},\nu_{\epsilon,\eta})) \ \ \forall (t,x,\mu)
\end{equation}
and so we have that
\begin{align} \label{eq:estcomp-2}
	V_2(t,x,\mu)-&V_2(t_{\epsilon,\eta},x_{\epsilon,\eta},\mu_{\epsilon,\eta})\\
	\geq&\frac{1}{2\epsilon}[\rho_\epsilon^2 -d^2_x((t,x,\mu),(s_{\epsilon,\eta},y_{\epsilon,\eta},\nu_{\epsilon,\eta}))]\\\nonumber
	\geq&\frac{1}{2\epsilon}\left[ (|t_{\epsilon,\eta}-s_{\epsilon,\eta}|^2-|t-s_{\epsilon,\eta}|^2)\right]+\frac{1}{2\epsilon}\left[|x_{\epsilon,\eta}-y_{\epsilon,\eta}|^2-|x-y_{\epsilon,\eta}|^2\right] +\\
	&+\frac{1}{2\epsilon}\left[ -W_2^2(\mu,\nu_{\epsilon,\eta})+W^2_2(\nu_{\epsilon,\eta},\mu_{\epsilon,\eta}))\right]\nonumber\\
    =&\frac{1}{\epsilon} (s-s_{\epsilon,\eta})(s_{\epsilon,\eta} -t_{\epsilon,\eta})+\frac{1}{2\epsilon} (s-s_{\epsilon,\eta})^2+\frac{1}{2\epsilon}\langle y-x_{\epsilon,\eta},y_{\epsilon,\eta}-x_{\epsilon,\eta}\rangle\nonumber.
\end{align}

Choosing $\pi_h$ as a variation in $\mu_{\epsilon,\eta}$, $\pi \in \Pi_o(\nu_{\epsilon,\eta},\mu_{\epsilon,\eta})$ and  $\mu_h:=pr_2\sharp \pi_h$ we have that
\begin{align*} 
	V_2(t,x,\mu)-&V_2(t_{\epsilon,\eta},x_{\epsilon,\eta},\mu_{\epsilon,\eta})\\
	\geq&\frac{1}{\epsilon} (t-t_{\epsilon,\eta})(t_{\epsilon,\eta} -s_{\epsilon,\eta})+\frac{1}{2\epsilon} (t-t_{\epsilon,\eta})^2+\frac{1}{2\epsilon}\langle x-x_{\epsilon,\eta},x_{\epsilon,\eta}-y_{\epsilon,\eta}\rangle\nonumber \\
	&+ o(|y_{\epsilon,\eta}-x_{\epsilon,\eta}|^2)- \frac{1}{\epsilon} \int \langle y-x,p_{\nu_{\epsilon,\eta},\mu_{\epsilon,\eta}^{\pi}}\rangle \,d{\pi}_h(x,y)
\end{align*}
Hence
\begin{equation}
	\left( -\frac{1}{\epsilon}(t_{\epsilon,\eta}-s_{\epsilon,\eta}),-\frac{x_{\epsilon,\eta}-y_{\epsilon,\eta}}{\epsilon}, \frac{1}{\epsilon}p^{\pi}\right) \in \partial^-V_2(t_{\epsilon,\eta},x_{\epsilon,\eta},\mu_{\epsilon,\eta}).
\end{equation}
Since that $V_1$ and $V_2$ are subsolution and supersolution, respectively, and $t_{\epsilon,\eta},s_{\epsilon,\eta} \neq T$, we have
\begin{align*}
&\frac{1}{\epsilon} (s_{\epsilon,\eta}-t_{\epsilon,\eta}) -\eta + H(s_{\epsilon,\eta},y_{\epsilon,\eta},\nu_{\epsilon,\eta}, \frac{y_{\epsilon,\eta}-x_{\epsilon,\eta}}{\epsilon},\frac{1}{\epsilon}p^{\bar\pi})\geq 0\\
&\frac{1}{\epsilon} (s_{\epsilon,\eta}-t_{\epsilon,\eta}) + H(t_{\epsilon,\eta},x_{\epsilon,\eta},\mu_{\epsilon,\eta}, \frac{y_{\epsilon,\eta}-x_{\epsilon,\eta}}{\epsilon},\frac{1}{\epsilon}p^{\pi})\leq 0
\end{align*}
Choosing ${\pi}=\bar{\pi}^{-1}$ and using Proposition \ref{prop:estimateH} we get
\begin{align*}
	-\eta &\geq H(t_{\epsilon,\eta},x_{\epsilon,\eta},\mu_{\epsilon,\eta}, \frac{y_{\epsilon,\eta}-x_{\epsilon,\eta}}{\epsilon},\frac{1}{\epsilon}p^{\bar{\pi}^{-1}})-H(s_{\epsilon,\eta},y_{\epsilon,\eta},\nu_{\epsilon,\eta}, \frac{y_{\epsilon,\eta}-x_{\epsilon,\eta}}{\epsilon},\frac{1}{\epsilon}p^{\bar{\pi}})\\
	&\geq -\rho_\epsilon(C+M\frac{|y_{\epsilon,\eta}-x_{\epsilon,\eta}|}{\epsilon}+\frac{1}{\epsilon}W_2(p^{\bar{\pi}^{-1}},p^{\bar\pi}))
\end{align*}
Using \eqref{eq:estrho} we get
\begin{align*}
-\eta \geq&-2\epsilon (\eta+\mathrm{Lip}(V_1))(C +M\frac{|y_{\epsilon,\eta}-x_{\epsilon,\eta}|}{\epsilon}+\frac{1}{\epsilon}W_2(p^{\bar{\pi}^{-1}},p^{\bar\pi})\\
\geq&-2(\eta+\mathrm{Lip}(V_1))\rho_\epsilon\geq -4\epsilon (\eta+\mathrm{Lip}(V_1))^2,
\end{align*}
which contradicts the property (c) of $\eta$.
\end{proof}

\appendix
\section{Some technical lemmas}

\begin{lemma}\label{lemma:absfil}
Let $\mathbb X,\mathbb Y,\mathbb Z$ be metric spaces.
Let $T:\mathbb X\times \mathbb Y\times\mathbb X\to (\mathbb X\times \mathbb Y)\times (\mathbb X\times \mathbb Y)$, $E:\mathbb X\times\mathbb Y\to \mathbb Z$ be Borel maps.
Given $\boldsymbol\xi\in\mathscr P(\mathbb X\times \mathbb Y\times\mathbb X)$,
define $\boldsymbol\eta=\mathrm{pr}_{12}\sharp\boldsymbol\xi$, $\boldsymbol\pi=\mathrm{pr}_{13}\sharp\boldsymbol\xi$, and
$\boldsymbol{\hat\eta}=(\mathrm{pr}_2\circ T)\sharp \boldsymbol\xi$.
Then
\[W_2(E\sharp\boldsymbol\eta,E\sharp\boldsymbol{\hat\eta})\le \|E\circ\mathrm{pr}_{12}-E\circ\mathrm{pr}_2\circ T\|_{L^2_{\boldsymbol\xi}}.\]
Moreover, if $T$ is continuous we have $\mathrm{supp}(\boldsymbol{\hat\eta})\subseteq \overline{\mathrm{pr}_2\circ T(\mathrm{supp}\,\boldsymbol\xi)}$.
\end{lemma}
\begin{proof}
See e.g. Section 5.2 in \cite{AGS}. Notice that, starting from given $\boldsymbol\eta\in\mathscr P(\mathbb X\times \mathbb Y)$, $\boldsymbol\pi\in \mathscr P(\mathbb X\times \mathbb X)$ such that $\mathrm{pr}_1\sharp\boldsymbol\eta=\mathrm{pr}_1\sharp\boldsymbol\pi$,
according to Lemma 5.3.2 in \cite{AGS} there exists $\boldsymbol\xi\in \mathscr P(\mathbb X\times \mathbb Y\times \mathbb X)$ such that $\boldsymbol\eta=\mathrm{pr}_{12}\sharp\boldsymbol\xi$, $\boldsymbol\pi=\mathrm{pr}_{13}\sharp\boldsymbol\xi$.
\end{proof}

\begin{corollary}
Set $\mathbb X=\mathbb T^d$, $\mathbb Y=\Gamma_I$, and define the continuous map $T:\mathbb T^d\times \Gamma_I\times \mathbb T^d\to (\mathbb T^d\times \Gamma_I)\times(\mathbb T^d\times \Gamma_I)$ by
\[T(x_1,\gamma_1,x_2)=((x_1,\gamma_1),(x_2,\gamma_2)),\textrm{ where }\gamma_2(t)=\gamma_1(t)-x_1+x_2\textrm{ for all }t\in I.\]
Then given $\boldsymbol\eta\in\mathscr P(X\times Y)$ and $\boldsymbol\pi\in\mathscr P(X\times X)$ with $\mu:=\mathrm{pr}_1\sharp\boldsymbol\eta=\mathrm{pr}_1\sharp\boldsymbol\pi$,
there exists $\boldsymbol{\hat\eta}\in\mathscr P(X\times Y)$ satisfying
\begin{align*}
W^2_2(\boldsymbol\eta,\boldsymbol{\hat\eta})&\le\|(x_1-x_2,x_1-x_2)\|^2_{L^2_{\boldsymbol\xi}}=2\iint_{X\times X}|x_1-x_2|^2\,d\boldsymbol\pi(x_1,x_2),\\
W^2_2(e_t\sharp\boldsymbol\eta,e_t\sharp\boldsymbol{\hat\eta})&\le\|\gamma_1(t)-\gamma_2(t)\|^2_{L^2_{\boldsymbol\xi}}=\iint_{X\times X}|x_1-x_2|^2\,d\boldsymbol\pi(x_1,x_2),
\end{align*}
where we choose $\mathbb Z=X\times Y$ and $E=\mathrm{Id}_{\mathbb Z}$ in the first case, and $\mathbb Z=\mathbb X$ with $E=e_t$ in the second case.
\par\medskip\par
In particular, if $\boldsymbol\pi\in \Pi_o(\mu,\nu)$, we have $W_2(\boldsymbol\eta,\boldsymbol{\hat\eta})\le \sqrt 2\cdot W_2(\mu,\nu)$ and 
$W_2(e_t\sharp\boldsymbol\eta,e_t\sharp\boldsymbol{\hat\eta})\le W_2(\mu,\nu)$.
\end{corollary}

\begin{lemma}[Projection]\label{lemma:projection}
Let $\mathbb X,\mathbb Y$ be metric spaces, $\mathbb X$ separable, $h:\mathbb X\to\mathbb Y$ be continuous. Suppose that $\mathbb F:\mathbb X\rightrightarrows \mathbb Y$ is a weakly measurable set valued map with compact values.
Then there exists a Borel map $\mathrm{Pr}_{\mathbb F}:\mathbb X\to \mathbb Y$ such that 
\begin{align*}
\mathrm{Pr}_{\mathbb F}(x)\in \mathbb F(x),\textrm{ for all }x\in \mathbb X,&&\mathrm{dist}(h(x),\mathbb F(x))=d_{\mathbb Y}(h(x),\mathrm{Pr}_{\mathbb F}(x)).
\end{align*}
\end{lemma}
\begin{proof}
Define the map $f:\mathbb X\times \mathbb Y\to \mathbb R$ by setting
\[f(x,y):=d_Y(h(x),y)-\mathrm{dist}(h(x),\mathbb F(x)).\]
Notice that $f(x,\cdot)$ is continuous for all $x\in \mathbb X$. We prove that $f(\cdot,y)$ is measurable for every $y\in \mathbb Y$.
Indeed, it is enough to show that $x\mapsto \mathrm{dist}(h(x),\mathbb F(x))$ is measurable. 
Define $g(x,y)=\mathrm{dist}(y,\mathbb F(x))$. Then $g(\cdot,y)$ is measurable for every $y\in \mathbb Y$ by the weak measurability of $\mathbb F$ (see Theorem 3.5 in \cite{Him}),
and the map $g(\cdot,\mathbb F(x))$ is continuous. Thus $g$ is measurable on $\mathbb X\times \mathbb Y$. The map $x\mapsto \mathrm{dist}(h(x),\mathbb F(x))$ is the composition 
between the measurable map $g$ and the continuous map $x\mapsto (h(x),x)$, hence it is measurable.\par
By compactness of $\mathbb F(x)$ and continuity of the distance function, for every $x\in\mathbb X$ there exists $y\in \mathbb F(x)$ such that $f(h(x),y)=0$,
in particular we have that $0\in \{f(x,y):\, y\in \mathbb F(x)\}$ for all $x\in\mathbb X$.
By Filippov's Implicit Function Theorem (see e.g. Theorem 7.1 in \cite{Him}) there is a measurable selection $\mathrm{Pr}_{\mathbb F}(\cdot)$ of $\mathbb{F}(\cdot)$
such that $f(x,\mathrm{Pr}_{\mathbb F}(x))=0$ for all $x\in \mathbb X$.
\end{proof}


\textbf{Proof of the Proposition \ref{prop:DPP}.}
\begin{proof}
Let $(\bar{\xi},\bar{\mu})$ be an optimal trajectory for $V(t_0,\xi_0,\mu_0)$, then
\begin{align*}
	V(t_0,\xi_0,\mu_0)&=  \int_{t_0}^\tau \mathscr  L(\bar{\xi}(s),\bar{\mu}_s)\,ds +  \int_\tau^T \mathscr L(\bar{\xi}(s),\bar{\mu}_s)\,ds + \mathscr  G(\bar{\xi}(T),\bar{\mu}_T)\\
	&=\int_{t_0}^\tau \mathscr  L(\bar{\xi}(s),\bar{\mu}_s)\,ds+ J(\tau,\bar{\xi}(\tau),\bar{\mu}_\tau,\bar{\xi}_{\shortmid{[\tau,T]}}, \bar{\mu}_{\shortmid{[\tau,T]}})\\
    &\geq \int_{t_0}^\tau \mathscr  L(\bar{\xi}(s),\bar{\mu}_s)\,ds+ V(\tau,\bar{\xi}(\tau),\bar{\mu}_\tau)\\
    &\geq \inf \left\{ \int_{t_0}^\tau \mathscr  L(\xi(s),\mu_s)\,ds + V(\tau,\xi(\tau),\mu_\tau), \ (\xi(\cdot),\mu_{\cdot}) \in \mathscr A_{[t_0,T]}(\xi_0,\mu_0)\right\},
\end{align*}
where we used the fact that $(\bar{\xi}_{\shortmid{[\tau,T]}}, \bar{\mu}_{\shortmid{[\tau,T]}})\in \mathscr A_{[\tau,T]}(\bar{\xi}(\tau),\bar{\mu}_\tau)$.
\par\medskip\par
Conversely, given any $(\xi(\cdot),\mu_\cdot)\in \mathscr A_{[t_0,T]}(\xi_0,\mu_0)$ and an optimal trajectory \par $(\hat{\xi},\hat{\mu})\in \mathscr A_{[\tau,T]}(\xi(\tau),\mu_\tau)$ for $V(\tau,\xi(\tau),\mu_\tau)$, 
we define
\begin{align*}
(\tilde{\xi}(t),\tilde{\mu}_t):=
\begin{cases}
({\xi}(t),{\mu}_t),&\textrm{ for }t\in[t_0,\tau],\\
(\hat{\xi}(t),\hat{\mu}_t)&\textrm{ for }t\in[\tau,T].
\end{cases}
\end{align*}
Since $(\tilde{\xi}(\cdot),\tilde{\mu}_\cdot) \in \mathscr A_{[t_0,T]}(\xi_0,\mu_0)$, recalling the optimality of $(\hat{\xi},\hat{\mu})$ we have 
\begin{align*}
V(t_0,\xi_0,\mu_0)&\leq J(t_0,\xi_0,\mu_0,\tilde{\xi}(\cdot),\tilde{\mu}_\cdot )=\int_{t_0}^\tau\mathscr L({\xi}(s),{\mu}_s)\,ds +\int_{t_0}^\tau \mathscr  L(\hat{\xi}(s),\hat{\mu}_s)\,ds+ \mathscr  G(\hat{\xi}(T),\hat{\mu}_T)\\
&=\int_{t_0}^\tau \mathscr  L({\xi}(s),{\mu}_s)\,ds+ V(\tau,\xi(\tau),\mu_\tau)
\end{align*}
Passing to the infimum on $(\xi(\cdot),\mu_\cdot)$, we get
\begin{align*}
V(t_0,\xi_0,\mu_0)\leq \inf \left\{ \int_{t_0}^\tau \mathscr  L(\xi(s),\mu_s)\,ds + V(\tau,\xi(\tau),\mu_\tau), \ (\xi(\cdot),\mu_\cdot)\in \mathscr A_{[t_0,T]}(\xi_0,\mu_0)\right\}.
\end{align*}
\end{proof}
\textbf{Proof of the Proposition \ref{prop:filippov}.}
\begin{proof}
	Let $\tau=\sup\{t\in [0,T]:\, \alpha^2:=2L^2\tau^2e^{2L\tau}(\tau+2)^2<1/2\}>0$ and $J=[0,\tau]$.
	We prove the result assuming that $\tau=T$ and so $I=J$. In the case $\tau<T$, we concatenate the trajectory $(\tilde\gamma,\boldsymbol{\tilde\mu})$ provided in $J$ by the statement 
	with any $(\hat\gamma,\boldsymbol{\hat\mu})\in\mathscr A_{[\tau,T]}(\tilde\gamma(\tau),\tilde\mu_{\tau})$ as follows
	\begin{itemize}
		\item define $\gamma=\tilde\gamma\oplus\hat\gamma$.
		\item given $\boldsymbol{\hat\eta}\in \boldsymbol{S_G}^{\boldsymbol\mu,\gamma}(\mu_\tau)$ with $\hat\mu_t=e_t\sharp\boldsymbol{\hat\eta}$, 
		we disintegrate 
		\begin{align*}
			\boldsymbol{\hat\eta}=\mu_\tau\otimes \hat\eta_y,&&\boldsymbol{\tilde\eta}=\mu_\tau\otimes\tilde\eta_y, 
		\end{align*}
		and define $\boldsymbol{\eta}\in\mathscr P(X\times\Gamma_I)$ by setting
		\[\iint_{X\times\Gamma_I} \varphi(x,\gamma)\,d\boldsymbol{\eta}(x,\gamma)=\int_X \left[\iint_{\Gamma_{[0,\tau]}\times \Gamma_{[\tau,T]}}
		\varphi(\gamma_1(0),\gamma_1\oplus\gamma_2)\,d\tilde\eta_y(x,\gamma_1)\,d\hat\eta_y(z,\gamma_2)\right]\,d\mu_{\tau}(y).\]
	\end{itemize}
	Indeed, notice that for $\mu_\tau$-a.e. $y\in X$ and  $\tilde\eta_y\otimes\hat\eta_y$-a.e.  $\left((x,\gamma_1),(z,\gamma_2)\right)\in (X\times\Gamma_{[0,\tau]})\times (X\times\Gamma_{[\tau,T]})$
	we have $\gamma_1(0)=x$, $\gamma_1(\tau)=\gamma_2(\tau)=y$. Moreover, set $\boldsymbol\mu=\{e_t\sharp\boldsymbol\eta\}_{t\in I}$, we have
	that $\boldsymbol\mu=\boldsymbol{\tilde\mu}\oplus\boldsymbol{\hat\mu}$, and $(\gamma,\boldsymbol\mu)$ and $\boldsymbol\eta$ satisfy all the requested properties.
	\par\medskip\par
	
	Thus from now on we assume without loss of generality that $\alpha^2:=2k^2T^2e^{2kT}(T+2)^2<1/2$.
	
	\emph{Claim 1: }There exists a constant $C>0$ such that for any $w(\cdot)\in AC(I;X)$ and $\boldsymbol\theta=\{\theta_t\}_{t\in I}\in C^0(I;\mathscr P_2(X))$,
	$x\in X$, $\gamma\in S_F^{\boldsymbol\theta,w}(x)$, $\xi\in S_G^{\boldsymbol\theta,w}(x)$, it holds 
	\begin{align*}
		|\xi(t)-\xi(0)|\le&C\left[1+|w(0)|+|\xi(0)|+\int_0^t \left(|w(s)-w(0)|+\mathrm{m}_2^{1/2}(\theta_s)\right)\,ds\right],\\
		|\gamma(t)-\gamma(0)|\le&C\left[1+|\gamma(0)|+\int_0^t \mathrm{m}_2^{1/2}(\theta_s)\,ds\right].
	\end{align*}
	\emph{Proof of Claim 1: }Define 
	\begin{align*}
		L=\max\{\mathrm{Lip}(F),\mathrm{Lip}(G)\},&&K=d_{G(0,0,0,\delta_0)}(0),&&\displaystyle C_1:=\max\left\{1,K,T,\dfrac{L}{2}T^2,L,e^{TL}\right\}.
	\end{align*}
	Recalling that $\dot\xi(s)\in G(s,w(s),\xi(s),\theta_s)$ for a.e. $s\in I$, we have
	\begin{align*}
		|\xi(t)-\xi(0)|\le&\int_0^t |\dot\xi(s)|\,ds\le\int_0^t \left(d_{G(0,0,0,\delta_0)}(\dot\xi(s))+K\right)\,ds\\
		\le&Kt+\mathrm{Lip}(G)\int_0^t \left(s+|w(s)|+|\xi(s)|+\mathrm{m}_2^{1/2}(\theta_s)\right)\,ds\\
		\le&\left(K+|w(0)|+|\xi(0)|\right)\cdot t+\dfrac{\mathrm{Lip}(G)}{2}t^2+ k\int_0^t \left(|w(s)-w(0)|+\mathrm{m}_2^{1/2}(\theta_s)\right)\,ds+\\&+\mathrm{Lip}(G)\int_0^s |\xi(s)-\xi(0)|\,ds.
	\end{align*}
	Gr\"onwall's inequality then implies
	\begin{align*}
		|\xi(t)-\xi(0)|\le&\left[\left(K+|w(0)|+|\xi(0)|\right)\cdot t+\dfrac{\mathrm{Lip}(G)}{2}t^2\right.+\\
		&\left.+\mathrm{Lip}(G)\int_0^t \left(|w(s)-w(0)|+\mathrm{m}_2^{1/2}(\theta_s)\right)\,ds\right]\cdot e^{t\cdot \mathrm{Lip}(G)}\\
		\le&\left[\left(1+|w(0)|+|\xi(0)|\right)\cdot T\max\{1,K\}+\dfrac{\mathrm{Lip}(G)}{2}T^2\right.+\\
		&\left.+\mathrm{Lip}(G)\int_0^t \left(|w(s)-w(0)|+\mathrm{m}_2^{1/2}(\theta_s)\right)\,ds\right]\cdot e^{T\cdot \mathrm{Lip}(G)}\\
		\le&\left[\left(1+|w(0)|+|\xi(0)|\right)\cdot C_1^2+C_1+C_1\int_0^t \left(|w(s)-w(0)|+\mathrm{m}_2^{1/2}(\theta_s)\right)\,ds\right]\cdot C_1\\
		\le&2C_1^3\left[1+|w(0)|+|\xi(0)|+\int_0^t \left(|w(s)-w(0)|+\mathrm{m}_2^{1/2}(\theta_s)\right)\,ds\right].
	\end{align*}
	The estimate on $\gamma$ is performed similarly, noticing that $F$ does not depend on $w(\cdot)$.\hfill$\diamond$
	\par\medskip\par
	\emph{Claim 2: }For any $(\bar x,\bar \mu)\in X\times \mathscr P_2(X)$, $w(\cdot)\in AC(I;X)$, $\boldsymbol\theta=\{\theta_t\}_{t\in I}\in C^0(I;\mathscr P_2(X))$, 
	$(\gamma,\boldsymbol\eta)\in S_F^{\boldsymbol\theta,w}(\bar x)\times \boldsymbol{S_G}^{\boldsymbol\theta,w}(\bar \mu)$, set $\mu_t=e_t\sharp\boldsymbol\eta$ it holds
	\[|\gamma(t)-\gamma(0)|+\mathrm{m}_2^{1/2}(\mu_t)\le 4C\left[1+|w(0)|+\mathrm{m}_2^{1/2}(\bar\mu)+|\bar x|+\int_0^t \left(|w(s)-w(0)|+\mathrm{m}_2^{1/2}(\theta_s)\right)\,ds\right],\]
	where $C>0$ is the constant appearing in Claim 1.
	\par\medskip\par
	\emph{Proof of Claim 2: }Follows by summing the first estimate  and the $L^2_{\boldsymbol\eta}$-norm of the second estimate in Claim 1.\hfill$\diamond$
	\par\medskip\par
	\emph{Claim 3: }Given $(\bar x,\bar\mu)\in X\times \mathscr P_2(X)$, consider the solution $z_{\bar x,\bar\mu}(\cdot)$ of the following Cauchy problem
	\begin{align*}\begin{cases}\dot z(t)=4Cz(t),\\ z(0)=4C(1+2|\bar x|+\mathrm{m}_2^{1/2}(\bar\mu)),\end{cases}\end{align*}
	where $C$ is the constant appearing in Claim 1, and define
	\begin{multline*}Q(\bar x,\bar \mu):=\{(w(\cdot),\boldsymbol\eta)\in C^0(I;X)\times\mathscr P_2(C^0(I;X)):\,\\ w(0)=\bar x, e_0\sharp\boldsymbol\eta=\bar\mu,
		|w(t)-w(0)|+\mathrm{m}_2^{1/2}(e_t\sharp\boldsymbol\eta)\le z_{\bar x,\bar\mu}(t)\textrm{ for all }t\in I\}.\end{multline*}
	Then $Q(\bar x,\bar \mu)$ is nonempty and closed, therefore is a complete metric subspace of $C^0(I;X)\times \mathscr P_2(C^0(I;X))$, and if $(w(\cdot),\boldsymbol\xi)\in Q(\bar x,\bar \mu)$, 
	we have 
	\[S_F^{\{e_t\sharp\boldsymbol\xi\}_{t\in I},w}(\bar x)\times \boldsymbol{S_G}^{\{e_t\sharp\boldsymbol\xi\}_{t\in I},w}(\bar \mu)\subseteq Q(\bar x,\bar \mu).\] 
	\par\medskip\par
	\emph{Proof of Claim 3: }We notice that set $\hat w(t)\equiv \bar x$ and $\boldsymbol{\hat \eta}=\bar\mu\otimes \delta_{\hat\gamma_x}$ where $\hat\gamma_x(t)=x$ for all $t\in I$,
	we have $(\hat w(\cdot),\boldsymbol{\hat\eta})\in Q(\bar x,\bar \mu)$. The closedness follow from the continuity of the map
	$w(\cdot)\mapsto |w(t)-\bar x|$ from $C^0(I;X)$ to $\mathbb R$, and from the continuity of $\boldsymbol\eta\mapsto e_t\sharp\boldsymbol\eta$.
	The last assertion follows from the definition of $z_{\bar x,\bar\mu}(\cdot)$ and from Claim 2.\hfill$\diamond$
	\par\medskip\par
	According to Theorem 9.7.2 in \cite{AuF}, set  $\overline B:= \overline{B_{\mathbb R^d}(0,1)}$, 
	there are $c>0$ and continuous maps $f:I\times X\times \mathscr P_2(X) \times \overline{B}\to \mathbb R^d$,
	$g:I\times X\times X\times\mathscr P_2(X) \times \overline{B}\to \mathbb R^d$ such that 
	\begin{itemize}
		\item $f$ and $g$ are parameterizations of $F$, $G$, respectively, i.e., 
		\begin{align*}F(t,x,\mu)=\{f(t,x,\mu,u):\, u\in \overline{B}\},&&G(t,x,y,\mu)=\{g(t,x,y,\mu,u):\, u\in \overline{B}\};\end{align*}
		\item for all $(t,u)\in I\times \overline{B}$ the maps $f(t,\cdot,\cdot,u)$ and $g(t,\cdot,\cdot,\cdot,u)$ are $cL$-Lipschitz continuous;
		\item for all $(t,x,\mu)\in I\times X\times \mathscr P_2(X)$, $y\in X$, $u,v\in \overline{B}$ it holds
		\begin{align*}
			|f(t,x,\mu,u)-f(t,x,\mu,v)|\le&c \max\{|h|:h\in F(t,x,\mu)\}\cdot |u-v|,\\
			|g(t,x,y,\mu,u)-g(t,x,y,\mu,v)|\le&c \max\{|h|:h\in G(t,x,y,\mu)\}\cdot |u-v|.
		\end{align*}
	\end{itemize}
	\par\medskip\par
	\emph{Claim 4: }
	Let $(w(\cdot),\boldsymbol\xi)\in Q(\bar x,\bar\mu)$ be fixed. Given $u\in\overline{B}$, denote by 
	$\gamma_{u}^{\bar x,\bar\mu,\boldsymbol\xi}(\cdot)\in S_F^{e_t\sharp\boldsymbol\xi,w}$ and $\xi_{x,u}^{\bar x,\bar\mu,w(\cdot),\boldsymbol\xi}(\cdot)\in S_G^{e_t\sharp\boldsymbol\xi,w}$ 
	the solutions of 
	\begin{align*}
		\begin{cases}\dot\gamma(t)=f(t,\gamma(t),e_t\sharp\boldsymbol\xi,u),\\ \gamma(0)=\bar x,\end{cases}&&
		\begin{cases}\dot\xi(t)=g(t,\xi(t),w(t),e_t\sharp\boldsymbol\xi,u),\\ \xi(0)=x,\end{cases}
	\end{align*}
	respectively. Then the maps $u\mapsto \gamma_{u}^{\bar x,\bar\mu,\boldsymbol\xi}$ and 
	$(x,u)\mapsto \xi_{x,u}^{\bar x,\bar\mu,w(\cdot),\boldsymbol\xi}(\cdot)$ are locally Lipschitz continuous.
	\par\medskip\par
	\emph{Proof of Claim 4: }
	For any $t\in I$ and $\gamma\in S_F^{w,\{e_t\sharp\boldsymbol\xi\}_{t\in I}}(\bar x)$, recalling Claim 3, we have 
	\begin{align*}
		|\gamma(t)-\bar x|\le\|z_{\bar x,\bar\mu}\|_{\infty},&&\mathrm{m}_2^{1/2}(e_t\sharp\boldsymbol\xi)\le \|z_{\bar x,\bar\mu}\|_{\infty}.
	\end{align*}
	Thus
	\[F(t,\gamma(t),e_t\sharp\boldsymbol\xi)\subseteq F(t,\bar x,\delta_0)+L\overline{B(0,2\|z_{\bar x,\bar\mu}\|_{\infty})},\]
	By continuity of $t\mapsto F(t,\bar x,\delta_0)$, there exits $K'_{\bar x}>0$ such that $F(t,\bar x,\delta_0)\subseteq \overline B(0,K'_{\bar x})$
	for all $t\in I$. Set $K_{\bar x,\bar \mu}:=K'_{\bar x}+2k\|z_{\bar x,\bar\mu}\|_{\infty}$, 
	it holds $F(t,\gamma(t),e_t\sharp\boldsymbol\xi)\subseteq \overline{B(0,K_{\bar x,\bar \mu})}$.
	Accordingly, given $u_1,u_2\in \overline{B}$, we have 
	\begin{align*}
		|\gamma_{u_1}^{\bar x,\bar\mu,\boldsymbol\xi}(t)-\gamma_{u_2}^{\bar x,\bar\mu,\boldsymbol\xi}(t)|
		\le&\int_0^t |f(s,\gamma_{u_1}^{\bar x,\bar\mu,\boldsymbol\xi}(s),e_s\sharp\boldsymbol\xi,u_1)\,ds-f(s,\gamma_{u_2}^{\bar x,\bar\mu,\boldsymbol\xi}(s),e_s\sharp\boldsymbol\xi,u_2)|\,ds\\
		\le&\int_0^t c(1+L+K_{\bar x,\bar\mu})\left(|u_1-u_2|+|\gamma_{u_1}^{\bar x,\bar\mu,\boldsymbol\xi}(s)-\gamma_{u_2}^{\bar x,\bar\mu,\boldsymbol\xi}(s)|\right)\,ds,
	\end{align*}
	and by Gr\"onwall's inequality
	\[|\gamma_{u_1}^{\bar x,\bar\mu,\boldsymbol\xi}(t)-\gamma_{u_2}^{\bar x,\bar\mu,\boldsymbol\xi}(t)|\le 
	Tc(1+L+K_{\bar x,\bar\mu})e^{Tc(1+k+K_{\bar x,\bar\mu})}\cdot |u_1-u_2|.\]
	Since the right hand side does not depend on $t$, we have the Lipschitz continuity of $u\mapsto \gamma_{u}^{\bar x,\bar\mu,\boldsymbol\xi}$.
	\par\medskip\par
	Fix now $x_0\in X$. For any $t\in I$ and $\xi\in S_G^{w,\{e_t\sharp\boldsymbol\xi\}_{t\in I}}(x)$ with $x\in\overline{B(x_0,1)}$, 
	recalling Claim 1 and Claim 3 
	\begin{align*}
		|\xi(t)-\xi(0)|\le C\left[1+|x_0|+1+z_{\bar x,\bar \mu}(t)\right],&&\mathrm{m}_2^{1/2}(e_t\sharp\boldsymbol\xi)\le \|z_{\bar x,\bar\mu}\|_{\infty}.
	\end{align*}
	By continuity of $t\mapsto F(t,x_0,\bar x,\delta_0)$, there exits $K''_{x_0,\bar x}>0$ such that 
	$G(t,x_0,\bar x,\delta_0)\subseteq \overline {B(0,K''_{x_0,\bar x})}$ for all $t\in I$.
	Set 
	\[\hat K_{x_0,\bar x,\bar \mu}:=K''_{x_0,\bar x}+L \left(C\left[1+|x_0|+1+\|z_{\bar x,\bar \mu}\|_{\infty}\right]+1+\|z_{\bar x,\bar \mu}\|_{\infty}\right),\]
	where $C$ is as in Claim 1, it holds 
	\begin{align*}
		G(t,\xi(t),&w(t),e_t\sharp\boldsymbol\xi)\subseteq\\ 
		\subseteq&G(t,x_0,\bar x,\delta_0)+L \left(|\xi(t)-\xi(0)|+|\xi(0)-x_0|+|w(t)-w(0)|+\mathrm{m}_2^{1/2}(e_t\sharp\boldsymbol\xi)\right)\overline{B(0,1)}\\
		\subseteq&\overline {B(0,K''_{x_0,\bar x})}+
		L \left(C\left[1+|x_0|+1+\|z_{\bar x,\bar \mu}\|_{\infty}\right]+1+\|z_{\bar x,\bar \mu}\|_{\infty}\right)\overline{B(0,1)}=
		\overline{B(0,\hat K_{x_0,\bar x,\bar \mu})}.
	\end{align*}
	Let $x_1,x_2\in \overline{B(0,1)}$, $u_1,u_2\in \overline{B}$ and consider
	\begin{align*}
		&|\xi_{x_1,u_1}^{\bar x,\bar\mu,w(\cdot),\boldsymbol\xi}(t)-\xi_{x_2,u_2}^{\bar x,\bar\mu,w(\cdot),\boldsymbol\xi}(t)|\le\\
		&\le |x_1-x_2|+\int_0^t |g(s,\xi_{x_1,u_1}^{\bar x,\bar\mu,w(\cdot),\boldsymbol\xi}(s),w(s),e_s\sharp\boldsymbol\xi,u_1)
		-g(s,\xi_{x_2,u_2}^{\bar x,\bar\mu,w(\cdot),\boldsymbol\xi}(s),w(s),e_s\sharp\boldsymbol\xi,u_1)|\,ds+\\
		&+\int_0^t |g(s,\xi_{x_2,u_2}^{\bar x,\bar\mu,w(\cdot),\boldsymbol\xi}(s),w(s),e_s\sharp\boldsymbol\xi,u_1)
		-g(s,\xi_{x_2,u_2}^{\bar x,\bar\mu,w(\cdot),\boldsymbol\xi}(s),w(s),e_s\sharp\boldsymbol\xi,u_2)|\,ds\\
		&\le |x_1-x_2|+L\int_0^t |\xi_{x_1,u_1}^{\bar x,\bar\mu,w(\cdot),\boldsymbol\xi}(s)-\xi_{x_2,u_2}^{\bar x,\bar\mu,w(\cdot),\boldsymbol\xi}(s)|\,ds+\int_0^t c\hat K_{x_0,\bar x,\bar \mu}\,ds\cdot |u_1-u_2|
	\end{align*}
	By Gr\"onwall's inequality
	\begin{align*}
		&|\xi_{x_1,u_1}^{\bar x,\bar\mu,w(\cdot),\boldsymbol\xi}(t)-\xi_{x_2,u_2}^{\bar x,\bar\mu,w(\cdot),\boldsymbol\xi}(t)|\le\\
		&\le |x_1-x_2|+\int_0^t |g(s,\xi_{x_1,u_1}^{\bar x,\bar\mu,w(\cdot),\boldsymbol\xi}(s),w(s),e_s\sharp\boldsymbol\xi,u_1)
		-g(s,\xi_{x_2,u_2}^{\bar x,\bar\mu,w(\cdot),\boldsymbol\xi}(s),w(s),e_s\sharp\boldsymbol\xi,u_1)|\,ds+\\
		&+\int_0^t |g(s,\xi_{x_2,u_2}^{\bar x,\bar\mu,w(\cdot),\boldsymbol\xi}(s),w(s),e_s\sharp\boldsymbol\xi,u_1)
		-g(s,\xi_{x_2,u_2}^{\bar x,\bar\mu,w(\cdot),\boldsymbol\xi}(s),w(s),e_s\sharp\boldsymbol\xi,u_2)|\,ds\\
		&\le e^{LT}(|x_1-x_2|+Tc\hat K_{x_0,\bar x,\bar \mu}\cdot |u_1-u_2|)\le e^{LT}(1+Tc\hat K_{x_0,\bar x,\bar \mu})(|x_1-x_2|+|u_1-u_2|).
	\end{align*}
	Since the last term does not depend on $t\in I$, by the arbitrariness of $x_0$ we have that 
	$(x,u)\mapsto \xi_{x,u}^{\bar x,\bar\mu,w(\cdot),\boldsymbol\xi}(\cdot)$ is locally Lipschitz continuous.\hfill$\diamond$
	\par\medskip\par
	Let $(w(\cdot),\boldsymbol\xi)\in Q(\bar x,\bar\mu)$. Given a Borel map $\bar v(\cdot)$ satisfying $\bar v(x)\in G(0,x,\bar x,\bar\mu)$ for all $x\in X$, by Filippov's implicit map theorem (see Theorem 7.1 in \cite{Him})
	there exists a Borel map $\bar u:X\to \overline{B}$ such that $\bar v(x)=g(0,x,\bar x,\bar \mu,\bar u(x))$ for all $x\in X$.
	As a consequence, the map $x\mapsto \xi_{x,u(x)}^{\bar x,\bar\mu,w(\cdot),\boldsymbol\xi}(\cdot)$ is Borel measurable and we can define
	$\boldsymbol\eta^{\bar v,\bar x,\bar\mu,w(\cdot),\boldsymbol\xi}:=\bar\mu \otimes \delta_{\xi_{x,\bar u(x)}^{\bar x,\bar\mu,w(\cdot),\boldsymbol\xi}}$,
	i.e., for all $\varphi\in C^0_b(X\times\Gamma_I)$ we set
	\[\int_{X\times\Gamma_I}\varphi(x,\gamma)\,d\boldsymbol\eta^{\bar v,\bar x,\bar\mu,w(\cdot),\boldsymbol\xi}(x,\gamma)
	=\int_X \varphi\left(x,\xi_{x,\bar u(x)}^{\bar x,\bar\mu,w(\cdot),\boldsymbol\xi}\right)\,d\bar\mu(x).\]
	Given $v_0\in F(0,\bar x,\bar\mu)$, let $u_0\in \overline{B}$ be such that $v_0=f(0,\bar x,\bar\mu,u_0)$. 
	\par\medskip\par
	\emph{Claim 5: }The map $(w(\cdot),\boldsymbol\xi)\mapsto (\gamma_{u_0}^{\bar x,\bar\mu,\boldsymbol\xi},\boldsymbol\eta^{\bar v,\bar x,\bar\mu,w(\cdot),\boldsymbol\xi})$ is a contraction in the complete metric space $Q(\bar x,\bar\mu)$
	and therefore it admits a unique fixed point.
	\par\medskip\par
	Let $(w_i(\cdot),\boldsymbol\xi_i)\in Q(\bar x,\bar\mu)$, $i=1,2$.
	Recalling that the map $e_s:X\times\Gamma_I\to X$ is $1$-Lipschitz continuous for all $s\in I$,
	\begin{align*}
		|\gamma_{u_0}^{\bar x,\bar\mu,\boldsymbol\xi_1}(t)-\gamma_{u_0}^{\bar x,\bar\mu,\boldsymbol\xi_2}(t)|
		\le&\int_0^t |f(s,\gamma_{u_0}^{\bar x,\bar\mu,\boldsymbol\xi_1}(s),e_s\sharp\boldsymbol\xi_1, u_0)
		-f(s,\gamma_{u_0}^{\bar x,\bar\mu,\boldsymbol\xi_2}(s),e_s\sharp\boldsymbol\xi_2, u_0)|\,ds\\
		\le&L\int_0^t \left(|\gamma_{u_0}^{\bar x,\bar\mu,\boldsymbol\xi_1}(s)-\gamma_{u_0}^{\bar x,\bar\mu,\boldsymbol\xi_2}(s)|+W_2(\boldsymbol\xi_1,\boldsymbol\xi_2)\right)\,ds
	\end{align*}
	Gr\"onwall's inequality yields
	\[|\gamma_{u_0}^{\bar x,\bar\mu,\boldsymbol\xi_1}(t)-\gamma_{u_0}^{\bar x,\bar\mu,\boldsymbol\xi_2}(t)|\le Lt e^{Lt}W_2(\boldsymbol\xi_1,\boldsymbol\xi_2).\]
	
	Given $x\in X$, we have
	\begin{align*}
		&|\xi_{x,\bar u(x)}^{\bar x,\bar\mu,w_1(\cdot),\boldsymbol\xi_1}(t)-\xi_{x,\bar u(x)}^{\bar x,\bar\mu,w_2(\cdot),\boldsymbol\xi_2}(t)|\\
		\le&\int_0^t |g(s,\xi_{x,\bar u(x)}^{\bar x,\bar\mu,w_1(\cdot),\boldsymbol\xi_1}(s),w_1(s),e_s\sharp\boldsymbol\xi_1,\bar u(x))-
		g(s,\xi_{x,\bar u(x)}^{\bar x,\bar\mu,w_2(\cdot),\boldsymbol\xi_2}(s),w_2(s),e_s\sharp\boldsymbol\xi_2,\bar u(x))|\,ds\\
		\le&L\int_0^t |\xi_{x,\bar u(x)}^{\bar x,\bar\mu,w_1(\cdot),\boldsymbol\xi_1}(s)-\xi_{x,\bar u(x)}^{\bar x,\bar\mu,w_2(\cdot),\boldsymbol\xi_2}(s)|\,ds
		+L\left[\int_0^t |w_1(s)-w_2(s)|\,ds+ W_2(\boldsymbol\xi_1,\boldsymbol\xi_2)\right].
	\end{align*}
	By Gr\"onwall's inequality
	\begin{align*}
		|\xi_{x,\bar u(x)}^{\bar x,\bar\mu,w_1(\cdot),\boldsymbol\xi_1}(t)-\xi_{x,\bar u(x)}^{\bar x,\bar\mu,w_2(\cdot),\boldsymbol\xi_2}(t)|
		\le&Le^{Lt}\left[\int_0^t |w_1(s)-w_2(s)|\,ds+ tW_2(\boldsymbol\xi_1,\boldsymbol\xi_2)\right]\\
		\le&Lt e^{Lt}\left[\|w_1-w_2\|_{\infty}+W_2(\boldsymbol\xi_1,\boldsymbol\xi_2)\right]\\
	\end{align*}
	We define a measure $\boldsymbol{\hat\pi}\in\mathscr P((X\times\Gamma_I),(X\times\Gamma_I))$ by setting 
	for all $\varphi\in C^0_b((X\times\Gamma_I)\times(X\times\Gamma_I))$
	\begin{multline*}
		\iint_{(X\times\Gamma_I)\times(X\times\Gamma_I)} \varphi((x_1,\xi_1),(x_2,\xi_2))\,d\boldsymbol{\hat\pi}((x_1,\xi_1),(x_2,\xi_2))=\\
		=\int_{X} \varphi\left(\left(x,\xi_{x,\bar u(x)}^{\bar x,\bar\mu,w_1(\cdot),\boldsymbol\xi_1}\right),\left(x,\xi_{x,\bar u(x)}^{\bar x,\bar\mu,w_2(\cdot),\boldsymbol\xi_2}\right)\right)\,d\bar\mu(x).
	\end{multline*}
	Notice that $\mathrm{pr}_i\sharp\boldsymbol{\hat\pi}=\boldsymbol\eta^{\bar v,\bar x,\bar\mu,w(\cdot),\boldsymbol\xi_i}$, $i=1,2$, therefore
	\begin{align*}
		W_2^2&(\boldsymbol\eta^{\bar v,\bar x,\bar\mu,w(\cdot),\boldsymbol\xi_1},\boldsymbol\eta^{\bar v,\bar x,\bar\mu,w(\cdot),\boldsymbol\xi_2})\\
		\le&\iint_{(X\times\Gamma_I)\times(X\times\Gamma_I)}\left[ |x_1-x_2|^2+\|\xi_1-\xi_2\|_{\infty}^2\right]\,d\boldsymbol{\hat\pi}((x_1,\xi_1),(x_2,\xi_2))\\
		=&\int_{X}|\xi_{x,\bar u(x)}^{\bar x,\bar\mu,w_1(\cdot),\boldsymbol\xi_1}(t)-\xi_{x,\bar u(x)}^{\bar x,\bar\mu,w_2(\cdot),\boldsymbol\xi_2}(t)|^2\,d\bar\mu(x)\\
		\le&L^2T^2e^{2LT}(T+1)^2\left[\|w_1-w_2\|_\infty+ W_2(\boldsymbol\xi_1,\boldsymbol\xi_2)\right]^2.
	\end{align*}
	We end up with 
	\begin{align*}
		\|\gamma_{u_0}^{\bar x,\bar\mu,\boldsymbol\xi_1}-\gamma_{u_0}^{\bar x,\bar\mu,\boldsymbol\xi_2}(t)\|_{\infty}^2&+
		W_2^2(\boldsymbol\eta^{\bar v,\bar x,\bar\mu,w(\cdot),\boldsymbol\xi_1},\boldsymbol\eta^{\bar v,\bar x,\bar\mu,w(\cdot),\boldsymbol\xi_2})\le\\
		\le&L^2T^2e^{2LT}(T+2)^2\left[\|w_1-w_2\|_\infty+ W_2(\boldsymbol\xi_1,\boldsymbol\xi_2)\right]^2\\
		\le&2L^2T^2e^{2LT}(T+2)^2\left[\|w_1-w_2\|^2_\infty+ W_2^2(\boldsymbol\xi_1,\boldsymbol\xi_2)\right]\\
		=&\alpha^2\cdot\left[\|w_1-w_2\|^2_\infty+ W_2^2(\boldsymbol\xi_1,\boldsymbol\xi_2)\right]\\
	\end{align*}
	and by assumption $0<\alpha<1$. By contraction's principle there is a unique fixed point $(\gamma,\boldsymbol\eta)\in Q(\bar x,\bar \mu)$.
	
	In particular, we have that $\dot\gamma(t)=f(t,\gamma(t),e_t\sharp\boldsymbol\eta,u_0)$, $\gamma(0)=\bar x$, in particular $\gamma\in C^1(\overline{I})$ and 
	$\dot\gamma(0)=f(0,\bar x,\bar \mu,u_0)=v_0$.
	On the other hand, for $\boldsymbol\eta$-a.e. $\xi\in \Gamma_I$ we have $\dot\xi(t)=g(t,\xi(t),\gamma(t),e_t\sharp\boldsymbol\eta,\bar u(\xi(0)))$,
	in particular, $\xi(\cdot)\in C^1(\overline{I})$ and $\dot\xi(0)=g(0,\xi(0),\bar x,\bar\mu,\bar u(\xi(0)))=\bar v(\xi(0))$.
\end{proof}


\begin{bibdiv}
\begin{biblist}

\bib{AGS}{book}{
  author ={Ambrosio, Luigi},
  author ={Gigli, Nicola},
  author ={Savar{\'e}, Giuseppe},
  title ={Gradient flows in metric spaces and in the space of probability measures},
  series ={Lectures in Mathematics ETH Z\"urich},
  edition={2},
  publisher={Birkh\"auser Verlag},
  place={Basel},
  date ={2008},
}

\bib{AC}{book}{
   author={Aubin, Jean-Pierre},
   author={Cellina, Arrigo},
   title={Differential inclusions},
   series={Grundlehren der Mathematischen Wissenschaften [Fundamental
   Principles of Mathematical Sciences]},
   volume={264},
   note={Set-valued maps and viability theory},
   publisher={Springer-Verlag, Berlin},
   date={1984},
}

\bib{AuF}{book}{
   author={Aubin, Jean-Pierre},
   author={Frankowska, H{\'e}l{\`e}ne},
   title={Set-valued analysis},
   series={Modern Birkh\"auser Classics},
   publisher={Birkh\"auser Boston Inc.},
   place={Boston, MA},
   date={2009},
}


\bib{AJZ}{article}{
 author = {Aussedat, A.}, Author = { Jerhaoui, O.}, Author = { Zidani, H.},
 title = {Viscosity solutions of centralized control problems in measure spaces},
 journal = {ESAIM, Control Optim. Calc. Var.},
 ISSN = {1292-8119},
 Volume = {30},
 Pages = {37},
 Note = {Id/No 91},
 Year = {2024}
}

\bib{AMQ}{article}{
    AUTHOR = {Averboukh, Yurii}, Author = {Marigonda, Antonio}, Author = { Quincampoix, Marc},
     TITLE = {Extremal shift rule and viability property for mean field-type
              control systems},
   JOURNAL = {J. Optim. Theory Appl.},
    VOLUME = {189},
      YEAR = {2021},
}

\bib{BadF}{article}{
 Author = {Badreddine, Zeinab},  Author = {Frankowska, H{\'e}l{\`e}ne},
 Title = {Solutions to {Hamilton}-{Jacobi} equation on a {Wasserstein} space},
 Journal = {Calc. Var. Partial Differ. Equ.},
 ISSN = {0944-2669},
 Volume = {61},
 Number = {1},
 Pages = {41},
 Note = {Id/No 9},
 Year = {2022},
}


\bib{BonF}{article}{
 Author = {Bonnet, Beno{\^{\i}}t  }, Author = {Frankowska, H{\'e}l{\`e}ne},
 Title = {Semiconcavity and sensitivity analysis in mean-field optimal control and applications},
 Journal = {J. Math. Pures Appl. (9)},
 ISSN = {0021-7824},
 Volume = {157},
 Pages = {282--345},
 Year = {2022},
}

\bib{CDLL}{book}{
   author={Cardaliaguet, Pierre},
   author={Delarue, Fran\c{c}ois},
   author={Lasry, Jean-Michel},
   author={Lions, Pierre-Louis},
   title={The master equation and the convergence problem in mean field
   games},
   series={Annals of Mathematics Studies},
   volume={201},
   publisher={Princeton University Press, Princeton, NJ},
   date={2019},
}


\bib{CQ}{article}{
 author={Cardaliaguet, P.},
 author={Quincampoix, M.},
 title={Deterministic differential games under probability knowledge of initial condition},
 journal={International Game Theory Review},
 volume={10},
 number={1},
 pages={1--16},
 date={2008},
 publisher={World Scientific, Singapore},
}

\bib{CMattain}{article}{
    AUTHOR = {Cavagnari, Giulia}, Author = {Marigonda, Antonio},
     TITLE = {Attainability property for a probabilistic target in
              {W}asserstein spaces},
   JOURNAL = {Discrete Contin. Dyn. Syst.},
    VOLUME = {41},
      YEAR = {2021},
      
}

\bib{CMQ}{article}{
    AUTHOR = {Cavagnari, Giulia}, Author = {Marigonda, Antonio}, Author = {Quincampoix,
              Marc},
     TITLE = {Compatibility of state constraints and dynamics for multiagent
              control systems},
   JOURNAL = {J. Evol. Equ.},
    VOLUME = {21},
      YEAR = {2021},    
}


\bib{CLOS}{article}{
    AUTHOR = {Cavagnari, Giulia}, Author = {Lisini, Stefano}, Author = {Orrieri, Carlo}, Author = {
              Savar\'{e}, Giuseppe},
     TITLE = {Lagrangian, {E}ulerian and {K}antorovich formulations of
              multi-agent optimal control problems: equivalence and
              gamma-convergence},
   JOURNAL = {J. Differential Equations},
    VOLUME = {322},
      YEAR = {2022},
     PAGES = {268--364},
}


\bib{G1}{article}{
 author={Gangbo, Wilfrid},
 author={Nguyen, Truyen},
 author={Tudorascu, Adrian},
 title={Hamilton-Jacobi equations in the Wasserstein space},
 journal={Methods and Applications of Analysis},
 volume={15},
 number={2},
 pages={155--184},
 date={2008},
 publisher={International Press of Boston, Somerville, MA},
}


\bib{G2}{article}{
 author={Gangbo, Wilfrid},
 author={Tudorascu, Adrian},
 title={On differentiability in the Wasserstein space and well-posedness for Hamilton-Jacobi equations},
 journal={Journal de Math{\'e}matiques Pures et Appliqu{\'e}es. Neuvi{\`e}me S{\'e}rie},
 volume={125},
 pages={119--174},
 date={2019},
 publisher={Elsevier (Elsevier Masson), Paris},
}



\bib{Him}{article}{
   author={Himmelberg, C. J.},
   title={Measurable relations},
   journal={Fund. Math.},
   volume={87},
   date={1975},
   pages={53--72},
   issn={0016-2736},
   doi={10.4064/fm-87-1-53-72},
}

\bib{JMQ}{article}{
author = {Jimenez, Chlo\'e},
author = {Marigonda, Antonio},
author = {Quincampoix, Marc},
title  = {Optimal control of multiagent systems in the Wasserstein space},
journal = {Calculus of Variations and Partial Differential Equations},
volume = {59 (2)},
number = {58},
pages  = {1--45},
year   = {2020},
}

\bib{Jimenez}{article}{
    AUTHOR = {Jimenez, Chlo\'{e}},
     TITLE = {Equivalence between strict viscosity solution and viscosity solution in the {Wasserstein} space and regular extension of the {Hamiltonian} in {{\(L^2_{\mathbb{P}}\)}}},
 Journal = {J. Convex Anal.},
 ISSN = {0944-6532},
 Volume = {31},
 Number = {2},
 Pages = {619--670},
 Year = {2024},
}


\bib{JMQ21}{article}{
    AUTHOR = {Jimenez, Chlo\'{e}}, Author = {Marigonda, Antonio}, Author = {Quincampoix,
              Marc},
     TITLE = {Dynamical systems and {H}amilton-{J}acobi-{B}ellman equations
              on the {W}asserstein space and their {$L^2$} representations},
   JOURNAL = {SIAM J. Math. Anal.},
    VOLUME = {55},
      YEAR = {2023},
    NUMBER = {5},
     PAGES = {5919--5966},
}	



\bib{LL}{article}{
 Author = {Lasry, Jean-Michel }, Author = { Lions, Pierre-Louis},
 Title = {Mean field games},
 Journal = {Jpn. J. Math. (3)},
 ISSN = {0289-2316},
 Volume = {2},
 Number = {1},
 Pages = {229--260},
 Year = {2007},
}


\bib{MQ}{article}{
 author={Marigonda, Antonio},
 author={Quincampoix, Marc},
 title={Mayer control problem with probabilistic uncertainty on initial positions},
 journal={Journal of Differential Equations},
 volume={264},
 number={5},
 pages={3212--3252},
 date={2018},
 publisher={Elsevier (Academic Press), San Diego, CA},
} 

\bib{Nad}{article}{
author={Nadler, S.B., Jr.}, 
title={Multi-valued contraction mappings},
journal={Pac. J. Math.},
date={1969}, 
volume={30}, 
pages={475--487}
}

\bib{San}{article}{
    author={Santambrogio, F.},
    title ={$\{$Euclidean, metric, and Wasserstein$\}$ gradient flows: an overview.},
    journal={Bull. Math. Sci.},
    volume={7},
    pages={87--154},
    date={2017},
    publisher={Springer},
    doi={doi.org/10.1007/s13373-017-0101-1},
}

\bib{Santambrogio}{book}{
   author={Santambrogio, Filippo},
   title={Optimal transport for applied mathematicians},
   series={Progress in Nonlinear Differential Equations and their
   Applications},
   volume={87},
   note={Calculus of variations, PDEs, and modeling},
   publisher={Birkh\"{a}user/Springer, Cham},
   date={2015},
}	


\bib{Villani}{book}{
AUTHOR = {Villani, C\'{e}dric},
     TITLE = {Optimal transport},
    SERIES = {Grundlehren der Mathematischen Wissenschaften [Fundamental
              Principles of Mathematical Sciences]},
    VOLUME = {338},
      NOTE = {Old and new},
 PUBLISHER = {Springer-Verlag, Berlin},
      YEAR = {2009},
}
\end{biblist}
\end{bibdiv}
\end{document}